\documentclass[a4paper,11pt,oneside,openany,reqno]{amsart}

\usepackage{latexsym,amsmath,amssymb,amsfonts}
\usepackage{bm}
\usepackage[dvipdfmx]{graphicx}
\usepackage{subfigure}
\usepackage{verbatim}
\usepackage{wrapfig}
\usepackage{ascmac}
\usepackage{cases}
\usepackage{amsthm}
\usepackage{color}

  \makeatletter
    
    \@addtoreset{equation}{section}
  \makeatother
\newtheorem{theo}{Theorem}[section]

\newtheorem{lemm}[theo]{Lemma}

\newcounter{c}
\newcounter{d}
\newcounter{b}
\newcommand{\cc}[1][]{\refstepcounter{c}#1\arabic{c}} 
\newcommand{\ce}[1][]{\refstepcounter{d}#1\arabic{d}}
\newcommand{\be}[1][]{\refstepcounter{b}#1\arabic{b}}
\newcommand{\res}{\mathop{\hbox{\vrule height 7pt width 
0.5pt depth 0pt \vrule height 0.5pt width 6pt depth 0pt}}\nolimits}

\newcommand{\R}{\mathbb{R}} 

\newcommand{\e}{\varepsilon} 
\newcommand{\spt}{\mathrm{spt}} 

\title[A diffused interface with the advection term]{A diffused interface with the advection term in a Sobolev space}
\date{}

\author{Yoshihiro Tonegawa}
\address{Department of Mathematics, Tokyo Institute of Technology,
152-8551, Tokyo, Japan}
\email{tonegawa@math.titech.ac.jp}

\author{Yuki Tsukamoto}
\address{Department of Mathematics, Tokyo Institute of Technology,
152-8551, Tokyo, Japan}
\email{tsukamoto.y.ag@m.titech.ac.jp}

\thanks{The first author is partially supported by JSPS KAKENHI Grant Numbers (A) 25247008 and (S) 26220702.}

\begin{document}

\maketitle
\begin{abstract}
We study the asymptotic limit of diffused surface energy in the van der Waals--Cahn--Hillard theory when an advection term is added and the energy is
uniformly bounded. We prove that the limit interface is an integral
varifold and the generalized mean curvature vector is determined by the advection term. 
As the application, a prescribed mean curvature problem is solved using the min-max method.
\end{abstract}

\section{Introduction}
The object of study in this paper is the energy functional appearing in the van der
Waals--Cahn--Hillard theory \cite{WC58,EG87}, 
\begin{eqnarray}
E_{\varepsilon}(u)= \int_\Omega \frac{\varepsilon|\nabla u|^2}{2}+\frac{W(u)}{\varepsilon},
\end{eqnarray}
where $u:\Omega\subset \mathbb{R}^n \to \mathbb{R}$ ($n \geq 2$) is the normalized density distribution of two phases of a material, $|\nabla u|^2=\sum_{k=1}^n(\partial u/\partial x_k)^2$
and $W:\mathbb{R} \to [0,\infty)$ is a double-well potential with two global minima at $\pm1$. In the thermodynamic context, $W$ corresponds
 to the Helmholtz free energy density and the typical example is
 $W(u)=(1-u^2)^2$. When the positive parameter $\e$ is small relative to
 the size of the domain $\Omega$ and $E_{\e}(u)$ is bounded, it is expected that 
 $u$ is close to $+1$ or $-1$ on most of $\Omega$ while a spatial change between $\pm 1$ 
 occurs within a hypersurface-like region of $O(\e)$ thickness which we may call the {\it diffused interface} of $u$. In this case, the quantity $E_{\e}(u)$ is expected to be proportional to 
 the surface area of the diffused interface. Due to the importance of the surface area in calculus of variations, it is interesting to investigate the validity of such expectation and other salient properties of $E_{\e}$. 
 
In this direction, there have been a number of works studying the asymptotic behavior of $E_{\e}$ as $\e\to 0+$ under various assumptions. 
For the energy minimizers with appropriate side conditions, 
it is well-known that it $\Gamma$-converges to the area
functional of the limit interface
\cite{KS,LM,LM87,MM,PS88}. On the other hand, due in part to the non-convex nature of the functional,
there may exist multiple and even infinite number of
critical points of $E_{\e}$ different from the energy minimizers. For general critical 
points, Hutchinson and the first author \cite{JH00} proved that
the limit is an integral stationary varifold \cite{WA72}. For general
{\it stable} critical points, 
the first author and Wickramasekera \cite{TW12} 
proved that the limit is an embedded 
real-analytic minimal hypersurface except for a closed singular set of codimension seven.
More recently, Guaraco \cite{Guaraco} showed that a uniform Morse index 
bound is sufficient to conclude the
same regularity for $n\geq 3$ and gave a new proof of Almgren-Pitts theorem \cite{Pitts}
as the application. The new proof significantly simplifies the existence part of the proof 
even though one needs to use Wickramasekera's hard regularity theorem \cite{Wick15}. 

While the investigations on the critical points of $E_\e$ have direct links to the minimal surface
theory as above, more generally, it turned out that suitable controls 
of the first variation of $E_{\e}$
guarantee the analogous good asymptotic behaviors. For example, under the 
assumption that 
\[ \liminf_{\e\to 0+} \big(E_{\e}(u_{\e})+\|f_{\e}\|_{W^{1,p}(\Omega)}\big)<\infty\]
with $f_{\e}:=-\e\Delta u_\e+W'(u_{\e})/\e$ and $p>n/2$, the first author \cite{TY02,TY05} proved that the limit interface is an integral varifold whose
generalized mean curvature belongs to $L^{q}$ ($q=p(n-1)/(n-p)>n-1$) with respect to the 
surface measure. Here $W^{1,p}(\Omega):=\{u\in L^p(\Omega)\,:\, \nabla u\in L^p(\Omega)\}$. The mean curvature of the limit interface is characterized by the weak $W^{1,p}$ limit 
of $f_{\e}$ \cite{RT}. 
Another example concerns one of De Giorgi's conjectures. Under the assumption that (with
$f_\e$ as above) 
\[
\liminf_{\e\to 0+}\big(E_{\e}(u_{\e})+\e^{-1}\|f_{\e}\|_{L^2(\Omega)}^2\big)<\infty\]
and $n=2,3$, R\"{o}ger-Sch\"{a}tzle \cite{RS} (independently \cite{NT} for the case of $n=2$) proved the similar
result. In this case, the limit interface has an $L^2$ generalized mean curvature. 

In this paper, along the line of research described above, we investigate the asymptotic 
behavior of $u_{\e}$ satisfying
\begin{equation}
- \varepsilon \Delta u_{\e} + \frac{W'(u_{\e})}{\varepsilon}= \varepsilon v_\e \cdot \nabla u_{\e}, \label{01}
\end{equation}
where $v_{\e}$ is considered here as a given vector field and we assume that
\[
\liminf_{\e\to 0+}\big(E_{\e}(u_{\e})+\|v_{\e}\|_{W^{1,p}(\Omega)}\big)<\infty
\]
and $p>n/2$. The problem is related to (parabolic) Allen-Cahn-type equations 
studied in \cite{LST1,KT16},
for example. It is also natural to investigate the effect of advection term
as $\e\to 0+$. We prove the analogous result Theorem \ref{maintheorem} to \cite{TY02,TY05}, namely,
the limit is an integral varifold with $L^q$ (the same as above) generalized
mean curvature which is characterized by the weak $W^{1,p}$ limit of $v_{\e}$. Using 
this result, we give some existence theorem for a {\it vectorial prescribed 
mean curvature problem}, as described in Theorem \ref{vmcp}. Despite the simplicity
of the problem, this is the first existence result in the 
setting of the min-max method, with minimal 
regularity assumptions on the prescribed vector field. 
As for the existence problem for {\it scalar} constant or prescribed mean curvature using
a min-max approach along the lines of Almgren-Pitts \cite{Pitts}, we mention 
papers by X. Zhou and J.J. Zhu \cite{ZZ1,ZZ2}. 

As for the proof, just as in the case of \cite{JH00,TY02,TY05},
the key point is to prove a certain monotonicity-type formula which is the essential 
tool in the setting of Geometric Measure Theory. We wish to treat $\e v_\e\cdot \nabla u_\e$ as a
perturbative term, and to do so, we need to control a certain ``trace'' norm of $v_\e$ 
on diffused interface. If an $\e$-independent upper density ratio estimate of diffused 
surface measure is available, then we can control $\e v_\e\cdot\nabla u_\e$ by the
$W^{1,p}$-norm of $v_\e$. For this purpose, we establish the key estimate,
Theorem \ref{theo5}, which gives a local uniform upper density ratio estimate.  
Once this part is done, the rest proceeds just like \cite{TY05} with minor
modifications. 

The paper is organized as follows. In Section 2 we state our assumptions and 
explain the main results. Section 3 contains the main estimates which ultimately give a monotonicity-type
formula, Theorem \ref{lemma6a}. In Section 4, we prove the main theorem by modifying
the proof in \cite{TY02,TY05}, and in Section 5, we give some concluding remarks.

\section{Assumptions and main results}

We use the notation that $U_r(a):=\{x\in\R^n\,:\, |x-a|<r\}$, $B_r(a):=\{x\in\R^n\,:\,
|x-a|\leq r\}$, $U_r:=U_r(0)$ and $B_r:=B_r(0)$. 
\subsection{Assumptions}
Throughout the paper, we assume that:
\begin{itemize}
\item[(a)] The function $W: \mathbb{R} \to [0,\infty)$ is $C^3$ and has two strict minima $W(\pm 1)=W'(\pm 1)=0$.
\item[(b)] For some $\gamma \in(-1,1)$, $W'>0$ 
on $(-1,\gamma)$ and $W'<0$ on $(\gamma,1)$.
\item[(c)] For some $\alpha \in (0,1)$ and $\kappa>0$, $W''(x)\geq \kappa$ for all $|x| \geq \alpha$.
\end{itemize}
Let $\Omega\subset\R^n$ be a bounded domain. We assume that we are given $W^{1,2}(\Omega)$ functions $\{u_i\}^\infty_{i=1}$, $W^{1,p}(\Omega; \mathbb{R}^n)$ vector fields $\{v_i\}^\infty_{i=1}$ and positive constants $\{\varepsilon_i\}^\infty_{i=1}$ satisfying
\begin{equation}
-\varepsilon_{i} \Delta u_{i} + \frac{W'(u_{i})}{\varepsilon_{i}} = \varepsilon_{i} v_{i} \cdot \nabla u_{i}\label{eq1}
\end{equation}
weakly on $\Omega$
for each $i\in\mathbb N$. In addition, assume that
\begin{equation}
\label{re-eq3}
\lim_{i \to \infty}\varepsilon_i=0, \ \ \  \frac{n}{2} < p<n
\end{equation}
and that there exist constants $c_{\cc[\label{c0}]}$, $E_0$ and $\lambda_0$ such that, for all
$i\in\mathbb N$, we have:
\begin{equation}
\|u_i\|_{L^{\infty}(\Omega)}\leq c_{\ref{c0}},  \label{eq1b}
\end{equation}
\begin{equation}
\int_{\Omega} \left(\frac{\varepsilon_i |\nabla u_i|^2}{2}+\frac{W(u_i)}{\varepsilon_{i}}\right) \leq E_0, \label{eq1c}
\end{equation}
\begin{equation}
\|v_i\|_{L^{\frac{np}{n-p}}(\Omega)}+\|\nabla v_i\|_{L^{p}(\Omega)} \leq \lambda_0.\label{eq1d}
\end{equation}

The condition \eqref{eq1b} is not essential and can be often derived from the PDE 
or the proof of existence. Here we assume \eqref{eq1b} for simplicity. 
Next, define
\begin{eqnarray*}
\Phi (s) := \int^s_{-1} \sqrt{W(t)/2} \ dt ,\,\,\,
w_i(x):= \Phi(u_i(x)).
\end{eqnarray*}
By the Cauchy--Schwarz inequality and (\ref{eq1c}), we obtain
\[
\int_{\Omega} |\nabla w_i| \leq \frac{1}{2}\int_{\Omega} \left(\frac{\varepsilon_i |\nabla u_i|^2}{2}+\frac{W(u_i)}{\varepsilon_i}\right) \leq \frac{1}{2}E_0 .
\]
Hence, by the compactness theorem for $BV$ functions \cite[Corollary 5.3.4]{WP72}, there exist a converging subsequence (which we denote by the same notation) $\{w_i\}$ in the $L^1$ norm
 and the limit BV function $w$. Define
\[
u (x): = \Phi^{-1}(w (x)).
\]
where $\Phi^{-1}$ is the inverse function of $\Phi$. It follows that $u_i$ converges to $u$ a.e. on $\Omega$. By Fatou's Lemma and (\ref{eq1c}), we have 
\[
\int_{\Omega} W(u) = \int_{\Omega} \lim_{i\to \infty} W(u_i) \leq \liminf_{i \to \infty} \int_{\Omega} W(u_i) =0.
\]
This shows that $u= \pm1$ a.e. on $\Omega$ and $u$ is a $BV$ function. For simplicity 
we write $\partial^* \{u=1\}$ as the reduced boundary \cite{WP72} of $\{u=1\}$ and $\|\partial^*\{u=1\}\|$
as the boundary measure.

\subsection{The associated varifolds}
We associate to each solution of (\ref{01}) a varifold in a natural way in the following.
We refer to \cite{WA72, LS83} for a comprehensive treatment of varifolds. 

Let ${\bf G}(n,n-1)$ be the Grassmannian, i.e.~the space of unoriented $(n-1)$-dimensional 
subspaces in $\mathbb{R}^n$.
We also regard $S \in {\bf G}(n,n-1)$ as the $n\times n$ matrix representing the orthogonal projection of $\mathbb{R}^n$ onto $S$. For two given square-matrices $S_1$ and $S_2$, we write 
$S_1\cdot S_2:={\rm trace}(S^t_1 \circ S_2)$, where the upper-script $t$ indicates the transpose of the matrix and $\circ$ is the matrix multiplication.
We say that $V$ is an $(n-1)$-dimensional varifold in $\Omega \subset \mathbb{R}^n$ if $V$
 is a Radon measure on ${\bf G}_{n-1}(\Omega):=\Omega \times {\bf G}(n,n-1)$. Let ${\bf V}_{n-1}(\Omega)$ be the set of all
 $(n-1)$-dimensional varifolds in $\Omega$. Convergence in the varifold sense means convergence in the
 usual sense of measures. For $V \in {\bf V}_{n-1}(\Omega)$, we let $\|V\|$ be the weight measure of $V$.
 For $V \in {\bf V}_{n-1}(\Omega)$, we define the first variation of $V$ by
\begin{equation}
\label{re-eq7}
\delta V(g) := \int_{{\bf G}_{n-1}(\Omega)} \nabla g(x) \cdot S \  dV(x,S)
\end{equation}
for any vector field $g \in C^1_c(\Omega;\mathbb{R}^n)$.
We let $\|\delta V\|$ be the total variation of $\delta V$. If $\|\delta V\|$ is 
absolutely continuous with respect to $\|V\|$, then the Radon-Nikodym
derivative $\delta V/\|V\|$ exists as a vector-valued $\|V\|$ measurable 
function. In this case, we define the generalized mean curvature vector of $V$ 
by $-\delta V/\|V\|$ and we use the notation $H_V$.

We associate to each function $u_i$ a varifold $V_i$ as follows.
First, we define a Radon measure $\mu_{i}$ on $\Omega$ by
\begin{equation}
 d \mu_i:= \frac{1}{\sigma}\Big(\frac{\varepsilon_i |\nabla u_i|^2}{2}+\frac{W(u_i)}{\varepsilon_i}\Big)d \mathcal{L}^n,
\label{defmu}
\end{equation}
where $\mathcal L^n$ is the $n$-dimensional Lebesgue measure and $\sigma:=\int_{-1}^1
\sqrt{2W(s)}\,ds$. 
Define $V_i \in {\bf V}_{n-1}(\Omega)$ by
\begin{equation}
V_i (\phi) := \int_{\{ |\nabla u_i| \neq 0 \}} \phi \Big( x,I-\frac{\nabla u_i(x)}{|\nabla u_i(x)|} \otimes
\frac{\nabla u_i(x)}{|\nabla u_i(x)|} \Big) d\mu_i(x)
\label{defvari}
\end{equation}
for $\phi \in C_c ({\bf G}_{n-1}(\Omega))$, where $I$ is the $n \times n$ identity matrix
and $\otimes$ is the tensor product of the two vectors. 
Note that $I-\frac{\nabla u_i(x)}{|\nabla u_i(x)|}\otimes\frac{\nabla u_i(x)}{|\nabla u_i(x)|}$ represents
the orthogonal projection to the $(n-1)$-dimensional subspace $\{ a\in\mathbb R^n\,:\, a\cdot\nabla u_i(x)=0\}$.
By definition, we have
\[
\|V_i\|=\mu_i\res_{\{|\nabla u_i|\neq 0\}}
\]
and by \eqref{re-eq7}, we have
\begin{equation}
\delta V_i (g) = \int_{\{ |\nabla u_i| \neq 0 \}} \nabla g \cdot
\Big( I-\frac{\nabla u_i}{|\nabla u_i|} \otimes
\frac{\nabla u_i}{|\nabla u_i|} \Big) d\mu_i
\label{defvari1}
\end{equation}
for each $g\in C^1_c(\Omega,\mathbb{R}^n)$. 
\subsection{Main Theorems}
With the above assumptions and notation, we show:
\begin{theo}
Suppose that $u_i,v_i,\e_i$ satisfy \eqref{eq1}-\eqref{eq1d} and let $V_i$ be the varifold associated with $u_i$ as in \eqref{defvari}. 
On passing to a subsequence we can assume that
\[
v_i \to v\ \mbox{weakly in} \ W^{1,p}, \ \ u_i \to u \ a.e., \ \ 
V_i \to V \mbox{ in the 
varifold sense.}
\]
Then we have the following properties. 
\begin{enumerate}
\item[(1)] For each $\phi \in C_c(\Omega)$,
\begin{eqnarray*}
\frac12\|V\|(\phi)&=&\lim_{i \to \infty}  \frac{1}{\sigma}\int_{\Omega}\frac{\varepsilon_i}{2}|\nabla u_i|^2 \phi = \lim_{i \to \infty} \frac{1}{\sigma}\int_{\Omega} \frac{W(u_i)}{\varepsilon_i} \phi \\
&=& \lim_{i \to \infty} \frac{1}{\sigma}\int_\Omega |\nabla w_i| \phi.
\end{eqnarray*}
\item[(2)] ${\rm spt}\,\|\partial^* \{u =1\}\| \subset {\rm spt}\,\|V\|$ and $\{u_i\}$ converges locally uniformly to $\pm1$
 on $\Omega\setminus {\rm spt}\,\|V\|$.
\item[(3)] For each $0<b<1$, $\{|u_i|\leq 1-b \}$ locally converges to 
${\rm spt}\,\|V\|$ in the Hausdorff distance sense in $\Omega$.
\item[(4)] $V$ is an integral varifold and the density $\theta(x)$ of $V$ satisfies
\begin{eqnarray*}
\theta(x)= \begin{cases}
{\mbox odd} & \mathcal{H}^{n-1} a.e. \  x\in \partial^*\{u=1\}, \\
{\mbox even} & \mathcal{H}^{n-1} a.e.  \ x\in {\rm spt}\,\|V\| \backslash \partial^*\{u=1\},
\end{cases}
\end{eqnarray*}
\item[(5)] the generalized mean curvature vector $H_V$ of $V$ is given by 
\begin{eqnarray*}
H_V(x)=(T_x\,{\rm spt}\,\|V\|)^{\perp}(v(x)),
\end{eqnarray*}
for $\|V\|$ a.e. $x\in\Omega$. 
\item[(6)] For $\tilde \Omega\subset\subset \Omega$, there exists a constant $\lambda_1$ depending only on $c_0$,$\lambda_0$,$n$,$p$,$W$,
$E_0$ and ${\rm dist}(\tilde \Omega,\,\partial \Omega)$ such that 
\[
\int_{\tilde \Omega}|H_V(x)|^{\frac{p(n-1)}{n-p}}\,d\|V\|(x)\leq 
\int_{\tilde \Omega}|v(x)|^{\frac{p(n-1)}{n-p}} \,d\|V\|(x)\leq \lambda_1.
\]
Note that $\frac{p(n-1)}{n-p}>n-1$ due to \eqref{re-eq3}. 
\end{enumerate}
\label{maintheorem}
\end{theo}
Since $V$ is integral and the generalized mean curvature vector is in the stated 
class, $V$ satisfies various good properties described in \cite[Section 17]{LS83}. 
In particular, ${\rm spt}\,\|V\|$ is a closed
countably $(n-1)$-rectifiable set (see \cite[17.9(1)]{LS83}), and writing $\Gamma:= {\rm spt}\,\|V\|$, for any 
$\phi\in C_c({\bf G}_{n-1}(\Omega))$, 
$$\int_{{\bf G}_{n-1}(\Omega)} \phi(x,S)\,dV(x,S)=\int_{\Gamma} \phi(x,T_x\,\Gamma)\theta(x)\,d\mathcal H^{n-1}(x).$$
Here, $T_x\,\Gamma\in {\bf G}(n,n-1)$ is the approximate tangent space of $\Gamma$ at 
$x$ which exists $\mathcal H^{n-1}$ a.e.~$x\in \Gamma$. With this notation, (5) implies
that $H_V(x)=(T_x\,\Gamma)^{\perp}(v(x))$ for $\mathcal H^{n-1}$ a.e.~$x\in \Gamma$, i.e.,
the generalized mean curvature vector of $V$ coincides with the projection of $v$
to the orthogonal subspace $(T_x\,\Gamma)^{\perp}$ for 
$\mathcal H^{n-1}$ a.e.~$x\in \Gamma$. We emphasize the difference of characterization of
the mean curvature vector from \cite{TY05}, where the similar equality holds only on 
the reduced boundary of $\{u=1\}$, while the equality here holds on the whole support of $\|V\|$ including
on the ``hidden boundary'' $\Gamma\setminus\partial^*\{u=1\}$.
If we additionally assume that $\theta=1$ for 
$\mathcal H^{n-1}$ a.e.~$x\in \Gamma$, then because of the integrability of $H_V$ and 
the Allard regularity theorem \cite{WA72}, 
except for a closed $\mathcal H^{n-1}$-null set, $\Gamma$ is locally
a $C^{1,2-\frac{n}{p}}$ hypersurface. Without the assumption $\theta=1$, we can 
still conclude that ${\rm spt}\,\|V\|$ is $C^{1,2-\frac{n}{p}}$ hypersurface
on a dense open set of ${\rm spt}\,\|V\|$, even though we do not know if the 
complement is $\mathcal H^{n-1}$-null or not. 
\subsection{A vectorial prescribed mean curvature problem}
As an application\footnote{The authors thank Nick Edelen for a discussion which inspired this application.}
of Theorem \ref{maintheorem} with suitable modifications, we prove the following:
\begin{theo}
Let $(M,g)$ be a smooth compact $n$-dimensional Riemannian manifold and let 
$\rho \in W^{2,p}(M)$ be a given function, where $p>\frac{n}{2}$. Then, there exists
a non-zero integral varifold $V$ in $M$ such that 
$$H_V(x)=(T_x \,{\rm spt}\,\|V\|)^{\perp}(\nabla \rho(x))$$ for $\|V\|$ a.e.~$x\in M$.
\label{vmcp}
\end{theo}
\begin{proof} We may assume $p<n$. 
Consider the following functional for $\e>0$ and $u\in W^{1,2}(M)$:
\[
F_{\e}(u):=\int_{M}\Big(\frac{\e|\nabla u|^2}{2}+\frac{W(u)}{\e}\Big)\exp(\rho)\,d{\omega}_g.
\]
By the Sobolev embedding, $\rho\in C^{0,2-\frac{n}{p}}(M)$ and thus 
$0<\exp(\min\,\rho)\leq \exp(\rho)\leq \exp(\max\,\rho)<\infty$. By considering the
path space in $W^{1,2}(M)$ connecting $u\equiv 1$ and $u\equiv -1$, the standard min-max method gives a non-trivial
critical point $u_\e$ for each $\e>0$, with uniform strictly positive lower and upper bounds of $F_{\e}(u_{\e})$. The 
critical point satisfies \eqref{eq1} with $v=\nabla\rho$ and $|u_\e|\leq 1$, with an appropriate modification of
the equation. 
Take a sequence $\e_i\to 0+$ and a corresponding min-max
critical points $u_i$. Then the sequence $u_i,\nabla\rho, \e_i$ satisfy all the
assumptions of Theorem \ref{maintheorem} with a small error terms coming from the metric of $M$
(see Guaraco's work \cite{Guaraco} for an adaptation of the results to Riemannian manifolds in the
case $v_\varepsilon=0$). 
The limit varifold $V$
thus has the desired property. 
\end{proof}
For more remarks on the main results, see Section 5. 
\section{The estimate for the upper density ratio}
In this section, we prove Theorem \ref{theo5}-\ref{lemm6}, which give $\e$-independent estimates of the upper and lower
density ratios of the energy. Throughout this 
section, we drop the index $i$ and set $\Omega=U_1=\{|x|<1\}$ since the result is local. Assume $u\in W^{1,2}(U_1)$ and $v\in W^{1,p}(U_1;\R^n)$ satisfy \eqref{eq1} with a positive
$\varepsilon$ and \eqref{eq1b}-\eqref{eq1d}
are satisfied for a given set of $c_{\ref{c0}}, E_0, \lambda_0$. The exponent $p$ satisfies \eqref{re-eq3}. 
%
We first derive two preliminary properties for $u$, Lemma \ref{prereg} and \ref{lemm8}. 
\begin{lemm}
\label{prereg}
There exists $c_{\cc[\label{c1}]}>0$ depending only on $c_{\ref{c0}},\lambda_0,n,p$ and $W$ such that
\begin{equation}
\label{u12a}
\sup_{x\in U_{1-\e}}\e |\nabla u(x)|\leq c_{\ref{c1}}
\end{equation}
and
\begin{equation}
\label{hu12a}
\sup_{x,x'\in U_{1-\e}} \e^{3-\frac{n}{p}}\frac{|\nabla u(x)-\nabla u(x')|}{|x-x'|^{2-\frac{n}{p}}}\leq 
c_{\ref{c1}}
\end{equation}
for $0<\e<1/2$. If $\e\geq 1/2$, then we have for any $0<s<1$
\begin{equation}
\label{u12b}
\sup_{x\in U_{s}} |\nabla u(x)|\leq c_{\ref{c1}}
\end{equation}
where $c_{\ref{c1}}$ depends additionally on $s$. 
In both cases, we have $u\in W_{loc}^{3,p}(U_1)$.
\end{lemm}
\begin{proof}
Consider the case $0<\e<1/2$. Define $\tilde u(x):=u(\e x)$ and $\tilde{v}(x):=\e v(\varepsilon x)$ for 
$x\in U_{\e^{-1}}$. After this change of variables, we obtain from \eqref{eq1} that
\begin{equation}
\label{re-eq2}
-\Delta \tilde{u}+W'(\tilde{u})= \tilde{v} \cdot \nabla \tilde{u} \ \ \ \mbox{weakly on }  U_{\e^{-1}}.
\end{equation}
Under the change of variables, we obtain from \eqref{eq1d}
\begin{equation}
\|\tilde v\|_{L^{\frac{np}{n-p}}(U_{\e^{-1}})}
+
\|\nabla \tilde v\|_{L^{p}(U_{\e^{-1}})}
\leq \lambda_0 \varepsilon^{2-\frac{n}{p}}.
\label{re-eq1}
\end{equation}
For any $U_2(x)\subset U_{\e^{-1}}$, let $\phi\in C^1_c(U_2(x))$ 
be a function such that $0\leq \phi\leq 1$, $\phi=1$ on $B_1(x)$ and $|\nabla\phi|\leq 4$ on $U_2(x)$. Use \eqref{re-eq2} with the test function $\tilde u\phi^2$. Using also \eqref{eq1b},
we obtain
\begin{equation}
\label{re-eq4}
\begin{split}
\int|\nabla\tilde u|^2\phi^2&\leq c_{\ref{c0}}\int (2\phi|\nabla\phi\|\nabla\tilde u|
+|W'|\phi^2+|\tilde v\|\nabla\tilde u|\phi^2)\\ 
&\leq \frac12\int|\nabla\tilde u|^2\phi^2+\int (4c_{\ref{c0}}^2|\nabla\phi|^2+c_{\ref{c0}}|W'|\phi^2
+c_{\ref{c0}}^2|\tilde v|^2\phi^2).
\end{split}
\end{equation} 
Since $\frac{np}{n-p}>2$, \eqref{re-eq1} and \eqref{re-eq4} give 
\begin{equation}
\label{re-eq5}
\sup_{B_2(x)\subset U_{\e^{-1}}}\int_{B_1(x)}|\nabla\tilde u|^2\leq c(c_{\ref{c0}},\lambda_0,n,p,W).
\end{equation}
We next note that the function $\tilde u\phi$ weakly satisfies the following equation:
\begin{equation}
\label{re-eq6}
-\Delta(\tilde u\phi)=-\tilde u\Delta\phi-2\nabla\phi\cdot\nabla\tilde u+(
\tilde v\cdot\nabla\tilde u-W'(\tilde u))\phi.
\end{equation} 
Using the standard $L^p$ theory \cite[Theorem 9.11]{NH77} to \eqref{re-eq6}, 
we may start a bootstrapping argument as follows. Staring with $q=2$, we have
\begin{equation*}
\begin{split}
&\nabla \tilde u\in L^q_{loc}\,\Longrightarrow\, \tilde v\cdot\nabla\tilde u
\in 
L_{loc}^{\frac{npq}{np+q(n-p)}}\,\Longrightarrow\, \tilde u\in W^{2,\frac{npq}{np+q(n-p)}}_{loc}\\
& \Longrightarrow\,\nabla \tilde u\in L_{loc}^{\frac{npq}{np-q(2p-n)}}
\end{split}
\end{equation*}
with the corresponding estimates relating these norms. Note that the exponent of integrability of $\nabla\tilde u$ is raised 
from $q$ to $q\cdot \frac{np}{np-q(2p-n)}$, with the factor strictly larger than one. Thus, in 
a finite number of bootstrapping, we obtain the $W^{2,s}_{loc}$ (with $s>n$) estimate for $\tilde u$, and by the Sobolev inequality, the $L^{\infty}_{loc}$ estimate for $\nabla\tilde u$. 
Again by the $L^p$ theory, we obtain the $W^{2,\frac{np}{n-p}}_{loc}$ estimate of $\tilde u$. 
In particular, by the Sobolev inequality, we obtain \eqref{u12a} and \eqref{hu12a}. 
Since the right-hand side of \eqref{re-eq6} is in $W^{1,p}_{loc}$ (note that
$\tilde v\cdot\nabla^2\tilde u\in L^{\frac{np}{2(n-p)}}_{loc}$ and 
$\frac{np}{2(n-p)}>p$ by \eqref{re-eq3}), we have $\tilde u\in W^{3,p}_{loc}$ and
the weak third-derivatives of $\tilde u$ exist. 
The case of $\e\geq 1/2$ does not require the change of variables as above and the proof is omitted. 
\end{proof}

\begin{lemm} 
Given $0<s<1$, there exist constants $0<\e_{\ce[\label{e1}]},\eta<1$ depending only on $c_0,\lambda_0,W,n,p$ and $s$ such that
\begin{eqnarray}
\sup_{x\in B_s}|u(x)| \leq 1+ \varepsilon^{\eta}
\end{eqnarray}
for $\e \leq \e_{\ref{e1}}$. \label{lemm8}
\end{lemm}
\begin{proof} 
Let $q = \frac{np}{n-p}-1$ and $\phi \in C^\infty_c(B_{\frac{s+1}{2}})$ with $\phi \geq0$. Multiplying (\ref{eq1})
 by $[(u-1)_+]^q \phi^2$, we have
\begin{eqnarray}
-\varepsilon \int q [(u-1)_+]^{q-1} |\nabla u|^2 \phi^2 +2 [(u-1)_+]^q \phi \nabla \phi \cdot \nabla u \nonumber \\
= \int \frac{W'}{\varepsilon}[(u-1)_+]^q \phi^2 - \int \varepsilon \nabla u\cdot v [(u-1)_+]^q \phi^2.
\end{eqnarray}
By $W'(u) \geq \kappa(u-1)$ for $u\geq 1$ and \eqref{u12a}, we obtain
\begin{eqnarray}
&&\frac{\kappa}{\varepsilon} \int [(u-1)_+]^{q+1} \phi^2 + \int \varepsilon q [(u-1)_+]^{q-1} |\nabla u|^2 \phi^2 \nonumber \\
&\leq& 2\varepsilon \int [(u-1)_+]^q \phi |\nabla \phi| |\nabla u| + c_{\ref{c1}} \int |v| [(u-1)_+]^q \phi^2 \nonumber \\
&\leq& \frac{q \varepsilon}{2} \int [(u-1)_+]^{q-1} |\nabla u|^2 \phi^2 + \frac{8\varepsilon}{q} \int
 [(u-1)_+]^{q+1} |\nabla \phi|^2 \nonumber \\ && + \frac{\kappa}{2\varepsilon} \int [(u-1)_+]^{q+1} \phi^2
+ \frac{\varepsilon^q c(q, c_{\ref{c1}})}{\kappa^q} \int |v|^{q+1} \phi^2,
\end{eqnarray}
which shows
\begin{eqnarray}
\frac{\kappa}{2\varepsilon} \int [(u-1)_+]^{q+1} \phi^2 \leq \frac{8\varepsilon}{q} \int
 [(u-1)_+]^{q+1} |\nabla \phi|^2 + \frac{\varepsilon^q c(q,c_{\ref{c1}})}{\kappa^q} \int |v|^{q+1} \phi^2. \nonumber \\
\end{eqnarray}
By \eqref{eq1b}, \eqref{eq1d} and iterating the computation above with suitable $\phi$, we obtain
\begin{eqnarray}
\int_{B_{s}} [(u-1)_+]^{q+1} \leq c_{\cc\label{c9}}(s,q,\lambda_0,n,p,W ,c_{\ref{c0}},c_{\ref{c1}}) \varepsilon^{q+1}.
\end{eqnarray}
To derive a contradiction, assume that $u(x_0)-1 \geq \varepsilon^\eta$ for some
 $x_0 \in B_{s}$. By \eqref{u12a}, for $y \in B_{\frac{\varepsilon^{1+\eta}}{2c_{\ref{c1}}}}(x_0)$,
\begin{eqnarray}
u(y)-1 \geq u(x_0)-1-\sup |\nabla u|\frac{\varepsilon^{1+\eta}}{2c_{\ref{c1}}} \geq \frac{\varepsilon^\eta}{2}.
\end{eqnarray}
Then we have 
\begin{eqnarray}
c_{\ref{c9}} \varepsilon^{q+1} \geq \int_{B_{\frac{\varepsilon^{1+\eta}}{2c_{\ref{c1}}}}(x_0)} [(u-1)_+]^{q+1}
\geq \left(\frac{\varepsilon^\eta}{2}\right)^{q+1} \omega_n \left(\frac{\varepsilon^{1+\eta}}{2c_{\ref{c1}}}\right)^n,
\end{eqnarray}
which show by $q = \frac{np}{n-p}-1$
\begin{eqnarray}
\varepsilon^{\eta \frac{np}{n-p} - \frac{np}{n-p} + n+n\eta} \leq c_{\cc}(s,q,\lambda_0,n,p,W,c_{\ref{c0}},c_{\ref{c1}}).
\end{eqnarray}
This is a contradiction if $\eta$ and $\e$ are sufficiently small. $u \geq -1-\varepsilon^\eta$ is proved similarly.
\end{proof}

The next Lemma \ref{lemm1aa} is the starting point of the ultimate establishment of
the monotonicity formula. 
\begin{lemm}
For $B_r(x) \subset U_1$, we have
\begin{equation}
\begin{split}
&\frac{d}{dr}\left\{ \frac{1}{r^{n-1}} \int_{B_r(x)} \left(\frac{\varepsilon |\nabla u|^2}{2}+\frac{W(u)}{\varepsilon}\right) \right\}
= \frac{1}{r^n} \int_{B_r(x)}  \left(\frac{W(u)}{\varepsilon}-\frac{\varepsilon |\nabla u|^2}{2} \right)  \\
& +\frac{\varepsilon}{r^{n+1}} \int_{\partial B_r(x)}\left((y-x) \cdot \nabla u \right)^2 + 
 \frac{\varepsilon}{r^n} \int_{B_r(x)} (v \cdot \nabla u)((y-x) \cdot \nabla u). 
\end{split}
\label{eq2a}
\end{equation}
\label{lemm1aa}
\end{lemm}
\begin{proof}
Multiply both sides of (\ref{eq1}) by $\nabla u \cdot g$, where $g = (g^1, \cdots,g^n) \in C^1_c (U_1; \mathbb{R}^n)$.
By integration by parts, we obtain 
\begin{eqnarray}
\int \Big(  \Big(\frac{\varepsilon |\nabla u|^2}{2}+\frac{W}{\varepsilon}\Big) \mathrm{div}g
-\varepsilon\sum_{i,j} u_{y_i} u_{y_j} g^i_{y_j}   + \varepsilon(v \cdot \nabla u)(\nabla u \cdot g)
\Big) =0.   \label{eq2}
\end{eqnarray}
We assume that $x=0$ after a suitable translation and let $g^j(y)=y_j \rho(|y|)$. Writing $r=|y|$, (\ref{eq2}) becomes
\begin{equation*}
\begin{split}
\int \Big(  &\Big(\frac{\varepsilon |\nabla u|^2}{2}+ \frac{W}{\varepsilon} \Big) \left(r \rho '  + n \rho \right)
- \varepsilon \frac{\rho '}{r} (y \cdot \nabla u)^2 \\ & - \varepsilon |\nabla u|^2 \rho + \varepsilon (\nabla u \cdot v)
(\nabla u \cdot y) \rho \Big) = 0.
\end{split}
\end{equation*}
We choose $\rho$ which is a smooth approximation of $\chi_{B_r}$, the characteristic
function of $B_r$, and then we take a limit $\rho \to \chi_{B_r}$. Then we have
\begin{equation*}
\begin{split}
&-(n-1)\int_{B_r} \left( \frac{\varepsilon |\nabla u|^2}{2}+\frac{W}{\varepsilon} \right)
+ r \int_{ \partial B_r} \left( \frac{\varepsilon |\nabla u|^2}{2}+\frac{W}{\varepsilon} \right)  \\
&= \int_{B_r} \left(\frac{W}{\varepsilon}-\frac
{\varepsilon|\nabla u|^2}{2}\right)+\frac{\varepsilon}{r} \int_{\partial B_r} (y \cdot \nabla u)^2 + \varepsilon \int_{B_r} (\nabla u \cdot v)( \nabla u \cdot y).
\end{split}
\end{equation*}
By dividing the above equation by $r^n$, the lemma follows.
\end{proof}

 We need the following lemma to control the negative contribution of the right-hand side
 of \eqref{eq2a}. 
\begin{lemm} 
Given $0<s<1$, there exist constants $0<\beta_{\be[\label{b1}]}<1$ and $0< \e_{\ce[\label{e2}]}<1$ which depend only on $c_{\ref{c0}}$, $\lambda_0$, $W$, $n$, $p$ and $s$ such that, if $\e\leq \e_
{\ref{e2}}$,
\begin{equation}
\sup_{B_{s}} \left( \frac{\varepsilon}{2} |\nabla u|^2 - \frac{W(u)}{\varepsilon}\right) \leq \varepsilon^{-\beta_{\ref{b1}}} .
\label{eq4}
\end{equation} \label{lemm2}
\end{lemm}
The proof of Lemma \ref{lemm2} is deferred to the end of this section. 
Next, 
for $B_r(x)\subset U_1$ and $0<r<{\rm dist}(x,\partial U_1)$, define
\[
E(r,x):= \frac{1}{r^{n-1}} \int_{B_r(x)} \left(\frac{\varepsilon |\nabla u|^2}{2}+\frac{W(u)}{\varepsilon}\right).
\]

Using Lemma \ref{lemm1aa} and Lemma \ref{lemm2}, 
we prove:
\begin{lemm} 
Given $0<s<1$, there exist constants $0<\e_{\ce[\label{e3}]},c_{\cc[\label{c4}]},c_{\cc[\label{c5}]}<1 $ which depend only on $c_0,\lambda_0,W,n,p$ and $s$ such that, if $B_{\varepsilon^{\beta_{\ref{b1}}}}(x) \subset  U_s$, 
$|u(x)|\leq \alpha$ and $\e \leq \e_{\ref{e3}}$, then
\begin{equation}
E(r,x)\geq c_{\ref{c4}} \ \ \ \mbox{for all} \  \varepsilon \leq r \leq c_{\ref{c5}} \varepsilon^{\beta_{\ref{b1}}}. \label{eq7}
\end{equation}
\label{lemma3}
\end{lemm}
\begin{proof}
By integrating (\ref{eq2a}) over [$\varepsilon,r$], we have
\begin{eqnarray}
E(r,x)-E(\varepsilon,x) \geq -\int^r_\varepsilon \frac{d\tau}{\tau^n} 
\int_{B_\tau(x)}\left( \frac{\varepsilon}{2} |\nabla u|^2 - \frac{W(u)}{\varepsilon}\right)_+ \nonumber\\
+\int^r_\varepsilon \frac{d\tau}{\tau^n} 
\int_{B_\tau(x)} \varepsilon  (\nabla u \cdot v)( \nabla u \cdot (y-x)). \label{eq8}
\end{eqnarray}
By (\ref{eq4}) and $B_r(x)\subset U_s$, we have
\begin{eqnarray}
\int^r_\varepsilon \frac{d\tau}{\tau^n} 
\int_{B_\tau(x)}\left( \frac{\varepsilon}{2} |\nabla u|^2 - \frac{W(u)}{\varepsilon}\right)_+
\leq \omega_n r  \varepsilon^{-\beta_1} . \label{eq9}
\end{eqnarray}
By \eqref{u12a} and \eqref{eq1d}, we have
\begin{equation}
\begin{split}
\Big|\int^r_\varepsilon \frac{d\tau}{\tau^n} 
\int_{B_\tau(x)} \varepsilon  (\nabla u \cdot v)( \nabla u \cdot (y-x))\Big| &\leq
\int^r_0 \frac{d\tau}{\tau^{n-1}} \int_{B_\tau(x)} c_{\ref{c1}}^2 \varepsilon^{-1} |v|  \\
&\leq  c(\lambda_0,n,p,c_{\ref{c1}}) r^{3-\frac{n}{p}} \varepsilon^{-1}. 
\end{split}
\label{eq10}
\end{equation}
Since $|u(x)| \leq \alpha$, using \eqref{u12a}, we have $|u(y)| \leq \frac{\alpha+1}{2}$ for all
 $y \in B_{\frac{(1-\alpha)\varepsilon}{2c_{\ref{c1}}}}(x)$. By choosing a larger $c_{\ref{c1}}$
 if necessary, we may assume $\frac{(1-\alpha)}{2c_{\ref{c1}}} \leq 1$. Define
\begin{eqnarray}
c_{\ref{c4}}:=\frac{\omega_n}{2} \frac{(1-\alpha)^n}{(2c_{\ref{c1}})^n} \min_{|t| \leq \frac{1+\alpha}{2}} W(t)>0. \nonumber
\end{eqnarray}
With this choice, we obtain
\begin{eqnarray}
E(\varepsilon,x) &\geq& \frac{1}{\varepsilon^{n-1}} 
\int_{B_{\frac{(1-\alpha)\varepsilon}{2c_{\ref{c1}}}}(x)} \frac{W(u)}{\varepsilon} \nonumber \\
&\geq& \omega_n \frac{(1-\alpha)^n}{(2c_{\ref{c1}})^n}
\min_{|t| \leq \frac{1+\alpha}{2}} W(t) = 2c_{\ref{c4}}. \label{eq11}
\end{eqnarray}
Since a larger $\beta_{\ref{b1}}$ satisfies \eqref{eq4} as well, we may assume $(3-\frac{n}{p})\beta_{\ref{b1}} -1 >0$ by choosing $\beta_{\ref{b1}}<1$ sufficiently close to 1. Then we may show that
the sum of \eqref{eq9} and \eqref{eq10} may be bounded from above by $c_{\ref{c4}}$ for 
sufficiently small $\e$ and $c_{\ref{c5}}$ if $r\leq c_{\ref{c5}}\e^{\beta_{\ref{b1}}}$. Then, 
(\ref{eq7}) follows from \eqref{eq8}-\eqref{eq11}.
\end{proof}

\begin{theo} 
\label{disc}
Given $0<s<1$, there exist constants $0<\e_{\ce\label{e4}},\beta_{\be \label{b2}} <1$ and $0<c_{\cc \label{c5a}}$ which depend only on $c_0$, $\lambda_0$, $W$, $n$, $p$ and $s$ such that, if $B_{r}(x) \subset U_s$, $c_{\ref{c5}} \varepsilon^{\beta_{\ref{b1}}}<r$ and $\e\leq \e_{\ref{e4}}$, then
\begin{equation}
\frac{1}{r^n} \int_{B_r(x)} \left( \frac{\varepsilon}{2} |\nabla u|^2 - \frac{W(u)}{\varepsilon}\right)_+ \leq 
\frac{c_{\ref{c5a}}}{r^{1-\beta_{\ref{b2}}}}(E(r,x)+1) .
\label{eq12}
\end{equation}
\end{theo}
\begin{proof}
The proof is similar to \cite[Proposition 3.4]{TY05} with a minor modification. Define $\beta_{\ref{b2}}:=\frac{1-\beta_{\ref{b1}}}{2\beta_{\ref{b1}}}$ and $\beta_{\be\label{b3}}:=\frac{1+\beta_{\ref{b1}}}{2}$. $\beta_{\ref{b2}}$
 and $\beta_{\ref{b3}}$ are chosen so that 
\begin{eqnarray}
\beta_{\ref{b1}} \beta_{\ref{b2}}= \beta_{\ref{b3}} - \beta_{\ref{b1}}, \\
0<\beta_{\ref{b2}}<1, \ \ \ \ 0<\beta_{\ref{b1}}<\beta_{\ref{b3}}<1.
\end{eqnarray}
We estimate the integral of (\ref{eq12}) by separating $B_r(x)$ into three disjoint sets. Define
\begin{eqnarray}
&\mathcal{A}& :=B_r(x) \backslash B_{r-\varepsilon^{\beta_{\ref{b3}}}}(x) \nonumber , \\
&\mathcal{B}&:=\{y\in B_{r-\varepsilon^{\beta_{\ref{b3}}}}(x)\, :\,  {\rm dist}(\{|u| \leq \alpha \},y)< \varepsilon^{\beta_{\ref{b3}}}  \} \nonumber , \\
&\mathcal{C}&:=\{y\in B_{r-\varepsilon^{\beta_{\ref{b3}}}}(x) \,:\, {\rm dist}(\{|u| \leq \alpha \},y) \geq \varepsilon^{\beta_{\ref{b3}}}  \} \nonumber .
\end{eqnarray}
Note that $r> c_{\ref{c5}}\varepsilon^{\beta_{\ref{b1}}}> \varepsilon^{\beta_{\ref{b3}}}$ for small $\e$.

The estimates of the integral over $\mathcal{A}$ and $\mathcal{B}$ are exactly the same as in \cite{TY05}. Namely, for $\mathcal A$, we use $\mathcal L^n(\mathcal A)\leq 
n\omega_n r^{n-1}\e^{\beta_{\ref{b3}}}$ and \eqref{eq4} as well as $r^{-1}\leq (c_{\ref{c5}}\e^{\beta_{\ref{b1}}})^{-1}$.
For $\mathcal B$, we use \eqref{eq7} and prove that $\mathcal L^n(\mathcal B)\leq c(n)\e^{n\beta_{\ref{b3}}}N$,
where $N$ is an integer satisfying $\mathcal B\subset \cup_{i=1}^N B_{5\e^{\beta_{\ref{b3}}}}(x_i)$. Here we only consider the estimate on $\mathcal{C}$ and refer the reader to the proof of \cite[Proposition 3.4]{TY05}. Define
\begin{eqnarray}
\phi(z) := \min \{1, \varepsilon^{-\beta_{\ref{b3}}} {\rm dist}(\{y\,:\,|y-x| \geq r \} \cup \{|u|\leq \alpha \},z ) \} \nonumber.
\end{eqnarray}
$\phi$ is a Lipschitz function and is 0 on $\{y\,:\,|y-x| > r \} \cup \{|u|< \alpha \}$, 1 on $\mathcal{C}$ and $|\nabla \phi|\leq
\varepsilon^{-\beta_{\ref{b3}}}$. Differentiate (\ref{eq1}) with respect to $x_j$, multiply it by $u_{x_j} \phi^2$ and 
sum over $j$. Then we have
\begin{equation}
\int \sum_{j} \varepsilon u_{x_j} \Delta u_{x_j} \phi^2= \int \frac{W''}{\varepsilon}|\nabla u|^2\phi^2 -
\e\sum_j\int(v\cdot\nabla u)_{x_j}\phi^2u_{x_j}. \label{eq1.20}
\end{equation}
By integration by parts, the Cauchy--Schwarz inequality and \eqref{u12a}, we obtain
\begin{eqnarray}
&&\int \varepsilon |\nabla^2 u |^2 \phi^2 +\frac{W''}{\varepsilon}|\nabla u |^2 \phi^2 \nonumber \\
&=& \int -\sum_{i,j} 2\varepsilon u_{x_j}u_{x_i x_j} \phi \phi_{x_i} - \varepsilon \nabla u \cdot v
(\Delta u \phi^2 +2 \phi \nabla u \cdot\nabla \phi) \nonumber \\
&\leq& \frac{1}{2} \int \varepsilon |\nabla^2 u |^2 \phi^2 +c_{\cc\label{c6}}\int (\varepsilon 
|\nabla u|^2|\nabla \phi|^2 + |v|^2 \phi^2 \varepsilon^{-1}),
\end{eqnarray}
where $c_{\ref{c6}}$ depends on $c_{\ref{c1}}$. 
Since $|u| \geq \alpha$ on the support of $\phi$, we have $W''\geq \kappa$. Thus
\begin{eqnarray}
\int \frac{\varepsilon}{2} |\nabla^2 u|^2 \phi^2 + \frac{\kappa}{\varepsilon}|\nabla u|^2 \phi^2
\leq c_{\ref{c6}} \int(\varepsilon 
|\nabla u|^2|\nabla \phi|^2 + |v|^2 \phi^2 \varepsilon^{-1}). \label{eq22a}
\end{eqnarray}
By $|\nabla \phi| \leq \varepsilon^{-\beta_{\ref{b3}}}$ and (\ref{eq1d}), we have
\begin{eqnarray}
\int \frac{\kappa}{\varepsilon}|\nabla u|^2 \phi^2 &\leq& 
c_{\ref{c6}} \left(\varepsilon^{-2\beta_{\ref{b3}}} \int_{B_r} \varepsilon |\nabla u|^2 +\varepsilon^{-1} \|v\|^2_{L^{\frac{np}{n-p}}}
(\omega_n r^n)^{\frac{np-2(n-p)}{np}}  \right) \nonumber \\
&\leq& c_{\cc\label{c7}} \left(\varepsilon^{-2\beta_{\ref{b3}}} \int_{B_r} \varepsilon |\nabla u|^2 + \varepsilon^{-1} r^{n-\frac{2(n-p)}{p}} \right), \label{eq23}
\end{eqnarray}
where $c_{\ref{c7}}=c_{\ref{c6}}+c(n,p)\lambda_0^2$. 
Since $\phi =1$ on $\mathcal{C}$, multiplying (\ref{eq23}) by $\frac{\varepsilon^2}{\kappa r^n}$,
\begin{eqnarray}
\frac{1}{r^n} \int_{\mathcal{C}} \frac{\varepsilon}{2} |\nabla u|^2 \leq \frac{c_{\ref{c7}}}{\kappa} \left(
\frac{\varepsilon^{2-2\beta_{\ref{b3}}}}{r} E(r,x) + \varepsilon r^{2-\frac{2n}{p}}\right).
\end{eqnarray}
By the definitions of $\beta_{\ref{b1}}$, $\beta_{\ref{b2}}$, $\beta_{\ref{b3}}$ and $r \geq c_{\ref{c5}} \varepsilon^{\beta_1}$, we have
\begin{eqnarray}
\frac{\varepsilon^{2-2\beta_{\ref{b3}}}}{r} \leq \frac{\varepsilon^{\beta_{\ref{b1}} \beta_{\ref{b2}}}}{r^{1-\beta_{\ref{b2}}}c_{\ref{c5}}^{\beta_{\ref{b2}}}},
\end{eqnarray}
and using $\varepsilon \leq r$,
\begin{eqnarray}
\varepsilon r^{2-\frac{2n}{p}} \leq \frac{1}{r^{\frac{2n}{p}-3}},
\end{eqnarray}
where $\frac{2n}{p}-3 <1$. Hence, we obtain
\begin{eqnarray}
\frac{1}{r^n} \int_{\mathcal{C}}\frac{\varepsilon}{2}|\nabla u|^2 \leq \frac{c_{\ref{c7}}}{\kappa}
\left(\frac{\varepsilon^{\beta_{\ref{b1}} \beta_{\ref{b2}}}}{r^{1-\beta_{\ref{b2}}}c_{\ref{c5}}^{\beta_2}}E(r,x)+
\frac{1}{r^{\frac{2n}{p}-3}} \right).
\end{eqnarray}
By re-defining $\beta_{\ref{b2}}= \min \{\beta_{\ref{b2}}, 4-\frac{2n}{p} \}$ and the estimates of
integrals over $\mathcal{A}$, $\mathcal{B}$ and $\mathcal{C}$, we proved \eqref{eq12}.
\end{proof}

To proceed, we need the following theorem (see \cite[Theorem 5.12.4]{WP72}).
\begin{theo}
Let $\mu$ be a positive Radon measure on $\mathbb{R}^n$ satisfying
\[
K(\mu):=\sup_{B_r(x)\subset\R^n} \frac{1}{r^{n-1}} \mu(B_r(x))< \infty.
\]
Then there exists a constant $c(n)$ such that
\[
\left|\int_{\mathbb{R}^n} \phi  \,d\mu \right| \leq c(n) K(\mu) \int_{\mathbb{R}^n} |\nabla \phi| \,d\mathcal{L}^n
\]
for all $\phi \in C^1_c(\mathbb{R}^n)$. \label{MZ}
\end{theo} 

\begin{theo} 
There exists a constant $0<c_{\cc \label{c8}}$ which depends only on $c_0$, $\lambda_0$, $W$, $E_0$, $n$ and $p$ such that, if $0<\e<1/2$ and $U_{2r}(x)\subset U_{1-\e}$, then
\begin{equation}
{\rm dist}(x, \partial U_{1-\e})^{n-1}\,E(r,x) \leq c_{\ref{c8}}.
\label{theo5eq}
\end{equation}
\label{theo5}
\end{theo}
\begin{proof}
Define
\begin{eqnarray}
E_1:=\sup_{U_{2r}(x)\subset U_{1-\e}} {\rm dist}(x, \partial U_{1-\e})^{n-1} E(r,x)\nonumber .
\end{eqnarray}
By Lemma \ref{prereg}, we have $\sup_{x\in U_{1-\e}}\e|\nabla u(x)|\leq c_{\ref{c1}}$. 
Thus for any $U_{2r}(x)\subset U_{1-\e}$, we have
\[
E(r,x)\leq \omega_n r(\frac{c_{\ref{c1}}^2}{2\e}+\sup_{|t|\leq c_{\ref{c0}}}\frac{W(t)}{\e})\leq
 \frac{c(n,c_{\ref{c1}},W)}{\e}
\]
and we have $E_1<\infty$ for each $\e$.  
In the following, we give an estimate of $E_1$ depending only on $c_0,\lambda_0,n,p,W$ and $E_0$. 
Let $U_{2r_0}(x_0)\subset U_{1-\e}$ be fixed such that
\begin{eqnarray}
{\rm dist}(x_0, \partial U_{1-\e})^{n-1}E(r_0,x_0)
> 
\frac{3}{4} E_1. \label{221}
\end{eqnarray}
For simplicity, define
\[
l:={\rm dist}\,(x_0,\partial U_{1-\e})=1-\e-|x_0|
\]
and change variables by $\tilde{x}=(x-x_0)/l$,
$\tilde{r}=r/l$, $\tilde\e=\e/l$, $\tilde{u}(\tilde x)=u(x)$  and 
$\tilde{v}(\tilde x)=lv(x)$. Note that $U_{l+\e}(x_0)\subset U_1$. 
In particular, we write
\[
\tilde r_0:=r_0/l\leq 1/2.
\]
By \eqref{eq1}, \eqref{eq1c} and \eqref{eq1d}, we have
\begin{equation}
\label{eq20a}
-\tilde{\varepsilon} \Delta \tilde{u} + \frac{W'(\tilde{u})}{\tilde{\varepsilon}} = \tilde{\varepsilon} 
\tilde{v} \cdot \nabla \tilde{u} \,\,\mbox{ for } \tilde x\in U_{1+\tilde\e},
\end{equation}
\begin{equation}
\int_{U_{1+\tilde\e}} \left(\frac{\tilde{\varepsilon}
 |\nabla \tilde{u}|^2}{2}+\frac{W(\tilde{u})}{\tilde{\varepsilon}}\right) \leq l^{1-n} E_0, \label{eq20}
\end{equation}
\begin{equation}
\|\tilde{v}\|_{L^{\frac{np}{n-p}}(U_{1+\tilde\e})}
+\|\nabla \tilde{v}\|_{L^{p}(U_{1+\tilde\e})}
 \leq l^{2-\frac{n}{p}}\lambda_0.  \label{eq19}
\end{equation}
Define for $B_{\tilde r}(\tilde x)\subset U_{1+\tilde\e}$ 
\begin{eqnarray}
\tilde{E}(\tilde{r},\tilde{x}):= \frac{1}{\tilde{r}^{n-1}} \int_{B_{\tilde{r}}(\tilde{x})} \left(\frac{\tilde{\varepsilon}
 |\nabla \tilde{u}|^2}{2}+\frac{W(\tilde{u})}{\tilde{\varepsilon}}\right).
\end{eqnarray}
Under the above change of variables, note that we have $E(r,x)=\tilde{E}(\tilde{r},\tilde{x})$. 
Next, for any $x\in B_{3l/4}(x_0)$, we have ${\rm dist}\,(x,\partial U_{1-\e})\geq l/4$. Hence
for any $x\in B_{3l/4}(x_0)$ and $r<l/8\leq {\rm dist}\,(x,\partial U_{1-\e})/2$, 
by the definition of $E_1$, we have $${\rm dist}(x,\partial U_{1-\e})^{n-1}E(r,x)\leq E_1.$$
This shows (again using ${\rm dist}(x,\partial U_{1-\e})\geq l/4$)
\begin{equation}
\sup_{\tilde{x}\in B_{\frac{3}{4}},\, 0<\tilde{r}<\frac{1}{8}} \tilde{E}(\tilde{r},\tilde{x}) \leq 4^{n-1}l^{1-n}E_1. \label{2.26a}
\end{equation}

We next let $c_{\ref{c1}}, c_{\ref{c4}},c_{\ref{c5}},c_{\ref{c5a}},\e_{\ref{e1}},\e_{\ref{e2}},\e_{\ref{e3}},\e_{\ref{e4}},\beta_{\ref{b1}},\beta_{\ref{b2}}$ be constants obtained in Lemma \ref{prereg}--\ref{lemma3} and Theorem \ref{disc}
corresponding to the same $c_{\ref{c0}},\lambda_0,n,p,W$ and $s=3/4$. Then note that 
the estimates up to Theorem \ref{disc} hold for $\tilde u$ and $\tilde v$ for $U_{3/4}$ and with respect to the new variables ($\tilde x$, $\tilde r$, $\tilde \e$ etc.) if 
\begin{equation}
\tilde \e\leq \hat \e:=\min\{\e_{\ref{e1}},\e_{\ref{e2}},\e_{\ref{e3}},\e_{\ref{e4}},1/2\}
\label{in1}
\end{equation} 
due to \eqref{eq20a} and \eqref{eq19}. It is important to note that $\hat \e$ depends only on
$c_0,\lambda_0,n,p$ and $W$. Note that \eqref{eq19} yields an upper bound for $\|\tilde v\|_{L^{\frac{np}{n-p}}
(U_{1+\tilde\varepsilon})}+\|\nabla\tilde v\|_{L^p(U_{1+\tilde\varepsilon})}$
independent of $l$, as $l<1$ and $2-\frac{n}{p}>0$. 
A priori, we do not know if \eqref{in1} holds or not and 
we prove the desired estimate for $E_1$ by exhausting all the possibilities. 

First consider the case $\tilde \e\geq \hat\e$. We use \eqref{u12b} and \eqref{u12a}, 
respectively, for $\tilde \e>1/2$ and $1/2\geq \tilde \e\geq \hat\e$. 
Suppose that $\tilde \e>1/2$. By \eqref{221} and \eqref{u12b}, we have
\begin{equation*}
\frac{3}{4l^{n-1}} E_1<\tilde E(\tilde r_0,0)\leq \omega_n\tilde r_0(\tilde \e c_{\ref{c1}}^2+2
\sup_{|x|\leq c_0}W(x))
\end{equation*}
and since $l\tilde\e=\e\leq 1$, $l<1$ and $\tilde r_0\leq 1/2$, we obtain
\begin{equation}
\label{eqto1}
E_1< \frac43 \omega_n(c_{\ref{c1}}^2+2\sup_{|x|\leq c_0} W(x)).
\end{equation}
If $1/2\geq \tilde \e\geq \hat \e$, again by \eqref{221}, \eqref{u12a} and $\tilde r_0\leq 1/2$, we have
\begin{eqnarray}
\frac{3}{4l^{n-1}} E_1 &<&
\tilde{E}(\tilde{r}_0,0)\leq \omega_n (c_{\ref{c1}}^2 +\sup_{|x|\leq c_0}W(x)) \frac{\tilde{r}_0}{\tilde{\varepsilon}} \nonumber \\
&\leq& \omega_n (c_{\ref{c1}}^2 +\sup_{|x|\leq c_0}W(x)) \frac{1}{2\hat{\varepsilon}} \nonumber
\end{eqnarray}
and we obtain
\begin{equation}
E_1<\omega_n (c_{\ref{c1}}^2 +\sup_{|x|\leq c_0}W(x)) \frac{2}{3\hat{\varepsilon}}. \label{m0}
\end{equation} 
Thus by \eqref{eqto1} and \eqref{m0}, if $\tilde \e\geq \hat\e$, $E_1$ is bounded by a constant which depends only on
$c_0,\lambda_0,n,p$ and $W$.

For the rest of the proof, consider the case $\tilde \e<\hat\e$ and consider the following four cases
(a)--(d) depending on how large $\tilde r_0=r_0/l$ is relative to $\tilde \e$ and $\tilde r_1$,
where $\tilde r_1$ will be determined shortly depending only on $c_0,\lambda_0,n,p,W$ and $E_0$: 
\begin{equation*}
(a)\,\tilde{r}_1 < \tilde{r}_0 \leq  \frac{1}{2},\,\,
(b)\,c_{\ref{c5}} \tilde{\varepsilon}^{\beta_{\ref{b1}}} < \tilde{r}_0 \leq \tilde{r}_1, \,\,
(c)\,\tilde{\varepsilon} < \tilde{r}_0 \leq c_{\ref{c5}} \tilde{\varepsilon}^{\beta_{\ref{b1}}},\,\,
(d)\, 0<\tilde{r}_0\leq \tilde{\varepsilon}.
\end{equation*}
To control the term involving $v$ in \eqref{eq2a}, define a Radon measure $$\mu(A) := \int_{A \cap B_{\frac{3}{4}}} \left(\frac{\tilde{\varepsilon}|\nabla \tilde{u}|^2}{2}+\frac{W(\tilde{u})}{\tilde{\varepsilon}}\right).$$ By Theorem \ref{MZ} and \eqref{2.26a} (note that \eqref{2.26a} has the restriction $\tilde r<1/8$ but this 
can be dropped easily by replacing $4^{n-1}$ by a larger constant depending only on $n$), we have
\begin{equation}
\Big|\int_{B_{\frac{3}{4}}} \phi\, d \mu \Big| \leq c(n) l^{1-n} E_1
\int_{\R^n} |\nabla \phi|\, d \mathcal{L}^n \label{eq17}
\end{equation}
for all $\phi \in C^1_c(\R^n)$. 
By (\ref{eq2a}) and (\ref{eq12}), if $c_{\ref{c5}} \tilde{\varepsilon}^{\beta_1}<\tilde{r} 
\leq \frac{1}{2}$, we have
\begin{equation}
\begin{split}
\frac{d}{d\tilde{r}} \tilde{E}(\tilde{r},0) \geq
&-\frac{c_{\ref{c5a}}}{\tilde{r}^{1-\beta_2}}(\tilde{E}(\tilde{r},0)+1) - \frac{1}{\tilde{r}^{n-1}} \int_{B_{\tilde{r}}} \tilde \e|\tilde{v}|
 |\nabla \tilde{u}|^2 \\
&+ \frac{1}{\tilde{r}^n} 
\int_{B_{\tilde{r}}} \left( \frac{W(\tilde{u})}{\tilde{\varepsilon}} -  \frac{\tilde{\varepsilon}}{2} |\nabla \tilde{u}|^2  \right)_+ .
\end{split}
\label{eq18}
\end{equation}
Let $\phi \in C^1_c(B_{\frac{3\tilde{r}}{2}})$ be such that $0\leq \phi \leq 1$, $\phi(y)=1$ in $B_{\tilde{r}}$ and $|\nabla \phi| \leq \frac{4}{\tilde{r}}$.
We use (\ref{eq17}) with (\ref{eq19}) by smoothly approximating $|\tilde{v}|$ as
\begin{equation}
\label{toch1}
\begin{split}
\int_{B_{\tilde{r}}} &\tilde\e|\tilde{v}|
|\nabla \tilde{u}|^2 \leq
\int_{B_{\frac{3\tilde{r}}{2}}} \tilde\e\phi |\tilde{v}|
|\nabla \tilde{u}|^2  \leq
c(n)l^{1-n}E_1
\int_{B_{\frac{3\tilde{r}}{2}}} | \nabla(\phi |\tilde{v}|)|   \\
&\leq c(n)l^{1-n}E_1
\int_{B_{\frac{3\tilde{r}}{2}}} \frac{4}{\tilde{r}}|\tilde{v}| + |\nabla \tilde{v}| 
\leq c(n) l^{3-n-\frac{n}{p}} \lambda_0 \tilde{r}^{n-\frac{n}{p}}E_1.
\end{split}
\end{equation}
Hence, for $c_{\ref{c5}}\tilde \e^{\beta_1}<\tilde r\leq \frac12$, \eqref{eq18} with \eqref{toch1}
and \eqref{2.26a} give
\begin{equation}
\begin{split}
\frac{d}{d\tilde{r}} \tilde{E}(\tilde{r},0) &\geq
-c_{\ref{c5a}}\tilde{r}^{\beta_2-1}(c(n)l^{1-n}E_1+1)- c(n) l^{3-n-\frac{n}{p}} \lambda_0 \tilde r^{1-\frac{n}{p}} E_1\\
&+\frac{1}{\tilde{r}^n} 
\int_{B_{\tilde{r}}} \left( \frac{W(\tilde{u})}{\tilde{\varepsilon}} -  \frac{\tilde{\varepsilon}}{2} |\nabla \tilde{u}|^2  \right)_+ . 
\end{split}\label{eq22}
\end{equation}
By integrating \eqref{eq22} over $\tilde r\in (\tilde{s}_1,\tilde{s}_2)$ with $c_{\ref{c5}}\tilde\e^{\beta_1}<\tilde s_1<\tilde s_2\leq \frac12$, we obtain
\begin{equation}
\begin{split}
\tilde{E}(\tilde s_2,0) - \tilde{E}(\tilde{s}_1,0)& \geq
-c_{\ref{ca}}(\tilde s_2^{\beta_2}+l^{2-\frac{n}{p}}\tilde s_2^{2-\frac{n}{p}})l^{1-n}E_1
-c_{\ref{ca}} \tilde s_2^{\beta_2}\\ &+\int_{\tilde s_1}^{\tilde s_2}\frac{d\tilde r}{\tilde{r}^n} 
\int_{B_{\tilde{r}}} \left( \frac{W(\tilde{u})}{\tilde{\varepsilon}} -  \frac{\tilde{\varepsilon}}{2} |\nabla \tilde{u}|^2  \right)_+ , 
\end{split}
\label{eq23a}
\end{equation}
where $c_{\cc\label{ca}}$ depends only on $c_0,\lambda_0,n,p$ and $W$. 
At this point, we choose $\tilde r_1<1/2$ depending only on $c_0,\lambda_0,n,p$ and $W$
so that 
\[
c_{\ref{ca}}(\tilde r^{\beta_2}_1+\tilde r_1^{2-\frac{n}{p}})
<\frac14.
\]
This in particular implies from \eqref{eq23a} that
if $c_{\ref{c5}}\tilde\e^{\beta_1}<\tilde s_1<\tilde s_2\leq \tilde r_1$, then
\begin{equation}
\tilde E(\tilde s_2,0)-\tilde E(\tilde s_1,0)\geq -c_{\ref{ca}}-\frac14 l^{1-n}E_1.
\label{eq22b}
\end{equation}
With this $\tilde r_1$ being fixed, we proceed to check that $E_1$ is bounded in terms
of $c_0,\lambda_0,n,p,W,E_0$ in each case (a)--(d).
\newline
Case (a): By \eqref{221}, \eqref{eq20} and $\tilde r_1<\tilde r_0$, we have
\begin{eqnarray}
\frac{3}{4} l^{1-n}E_1\leq \tilde{E}(\tilde{r}_0,0)\leq \tilde{r}_0^{1-n} l^{1-n}
E_0 \leq \tilde{r}_1^{1-n} l^{1-n} E_0. \nonumber
\end{eqnarray}
Hence 
\[
E_1\leq \frac43 \tilde r_1^{1-n}E_0
\]
and $E_1$ is bounded by a constant depending only on $c_0,\lambda_0,n,p,W$ and $E_0$.
\newline
Case (b): Since $c_{\ref{c5}}\tilde\e^{\beta_1}<\tilde r_0\leq \tilde r_1$, we may use \eqref{eq22b} with
$\tilde s_2=\tilde{r}_1$ and $\tilde s_1=\tilde{r}_0$. Then we obtain
\[
\tilde{E}(\tilde{r}_1,0) - \tilde{E}(\tilde{r}_0,0) \geq -c_{\ref{ca}} - \frac{1}{4} l^{1-n}E_1. 
\]
Then, by \eqref{221} and \eqref{eq20}, we obtain
\[
E_1\leq 4(\tilde E(\tilde r_1,0)+c_{\ref{ca}})l^{n-1}\leq 4(\tilde r_1^{1-n}E_0+c_{\ref{ca}}),
\]
which depends only on $c_0,\lambda_0,n,p,W$ and $E_0$. 
\newline
Case (c): By the same estimate used in the proof of Lemma \ref{lemma3}, we have
\begin{equation}
\tilde{E}(c_{\ref{c5}} \tilde{\varepsilon}^{\beta_1},0) - \tilde{E} (\tilde{r}_0,0) \geq -c_{\ref{c4}}.
\label{pa1}\end{equation}
We use \eqref{eq22b} with $\tilde s_1=c_{\ref{c5}}\tilde\e^{\beta_1}$ and $\tilde s_2=\tilde r_1$ to obtain 
\begin{equation}
\tilde{E}(\tilde{r}_1,0) - \tilde{E}(c_{\ref{c5}} \tilde{\varepsilon}^{\beta_1},0) \geq -c_{\ref{ca}}
 - \frac{1}{4} l^{1-n}E_1,
\label{pa2}\end{equation}
and \eqref{pa1} and \eqref{pa2} combined with \eqref{221} give
\[
E_1\leq 4l^{n-1} (\tilde E(\tilde r_1,0)+c_{\ref{c4}}+c_{\ref{ca}})\leq
4\tilde r_1^{1-n}E_0+4(c_{\ref{c4}}+c_{\ref{ca}}),
\]
which depends only on $c_0,\lambda_0,n,p,W$ and $E_0$. 
\newline
Case (d): Since $\tilde r_0\leq \tilde \e$, we use \eqref{u12a} to obtain
\begin{equation}\label{pa3}
\tilde{E}(\tilde{r}_0,0) \leq\omega_n (c_{\ref{c1}}^2 +\sup_{|x|\leq c_0}W(x)) \frac{\tilde{r}_0}{\tilde{\varepsilon}} 
\leq\omega_n (c_{\ref{c1}}^2 +\sup_{|x|\leq c_0}W(x)).
\end{equation}
Then \eqref{pa3} and \eqref{221} gives 
\[
E_1<\frac{4}{3}\omega_n (c_{\ref{c1}}^2 +\sup_{|x|\leq c_0}W(x))l^{n-1}\leq \omega_n (c_{\ref{c1}}^2 +\sup_{|x|\leq c_0}W(x)).
\]
This completes the estimate for $E_1$.
\end{proof}
Once we obtain the upper density estimate, we may obtain the following monotonicity 
formula. 
\begin{theo} 
Given $0<s<1$, there exist constants $0<c_{\cc \label{c10}}$ and $0<\e_{\ce\label{e6}}<1$ depending only on $c_0,\lambda_0,n,p,W,E_0$ and $s$, such that, if $c_{\ref{c5}} \e^{\beta_{\ref{b1}}}\leq s_1 <s_2$, $B_{s_2}(x)\subset U_s$ and $\e < \e_{\ref{e6}}$,
then
\begin{equation}
\begin{split}
E(s_2,x)-E(s_1,x) \geq &-c_{\ref{c10}} (s_2^{2-\frac{n}{p}}+s_2^{\beta_{\ref{b2}}}) \\
&+ \int^{s_2}_{s_1} \frac{d \tau}{\tau^n} 
\int_{B_{\tau}(x)} \Big( \frac{W(u)}{\e} -  \frac{\e}{2} |\nabla u|^2  \Big)_+. 
\end{split}
\label{eq41}
\end{equation}
\label{lemma6a}
\end{theo}
\begin{proof}
Let $\e_{\ref{e6}}=\min\{\e_{\ref{e1}},\e_{\ref{e2}},\e_{\ref{e3}},\e_{\ref{e4}},(1-s)/2\}$ corresponding
to the given $s$ and suppose that $\e<\e_{\ref{e6}}$. 
For any $x\in U_s$ and $0<r<(1-s)/2$, by Theorem \ref{theo5}, $E(r,x)\leq c_{\ref{c8}}
(1-s-\e)^{1-n}$, where the right-hand side is bounded by a constant depending only 
on $c_0,\lambda_0,n,p,W,E_0$ and $s$. For $B_{s_2}(x)\subset U_s$, we 
have \eqref{eq12} and \eqref{eq2a}. Arguing as \eqref{eq17}-\eqref{eq23a} without change of variables (so $l=1$) and
with $\mu$ restricted to $B_s$ in place of $B_{3/4}$, we obtain \eqref{eq41}.
\end{proof}
\begin{theo} 
Given $0<s<1$, there exist constants $0<c_{\cc \label{c9a}}$ depending only on 
$c_0,\lambda_0,n,p,W,E_0$ and $s$ such that, if $\e<\e_{\ref{e6}}$, $|u(x)|\leq 
\alpha$, $\varepsilon\leq r$ and 
$B_r(x)\subset U_s$, then
\begin{equation}
\label{eqlemm6}
E(r,x) \geq c_{\ref{c9a}}.
\end{equation} \label{lemm6}
\end{theo}
\begin{proof}
By Lemma \ref{lemma3}, we may assume $c_{\ref{c5}} \tilde{\e}^{\beta_{\ref{b1}}} \leq r$
and $E(c_{\ref{c5}}\e^{\beta_{\ref{b1}}},x)\geq c_{\ref{c4}}$. 
In \eqref{eq41}, let $s_1=c_{\ref{c5}}\e^{\beta_{\ref{b1}}}$ and $s_2=r$. Fix $r_1>0$
depending only on $c_0,\lambda_0,n,p,W,E_0$ and $s$ so that 
$c_{\ref{c10}}(r_1^{2-\frac{n}{p}}+r_1^{\beta_{\ref{b2}}})\leq c_{\ref{c4}}/2$.
Then for $c_{\ref{c5}}\e^{\beta_{\ref{b1}}}\leq r\leq r_1$,
\eqref{eq41} shows that $E(r,x)\geq c_{\ref{c4}}/2$. For $1>r>r_1$, $E(r,x)
\geq r_1^{n-1} E(r_1,x)\geq r_1^{n-1} c_{\ref{c4}}/2$. Thus, setting
$c_{\ref{c9a}}=r_1^{n-1}c_{\ref{c4}}/2$, we have \eqref{eqlemm6}.
\end{proof}

For the rest of the present section, we finish the proof of Lemma \ref{lemm2}. 
We use the following result proved in \cite[Lemma 3.9]{TY05}.
\begin{lemm} 
Given $0<\eta,\beta_{\be\label{b4}} <1$, $\eta \leq \beta_{\ref{b4}}$, $0<c_{\cc\label{c11}}$, there exist $\e_{\ce\label{e5}}>0$, $c_{\cc\label{c12}}>0$
 depending only on $\eta$, $\beta_{\ref{b4}}$, $c_{\ref{c11}}$, $n$ and $W$ with the following properties: Suppose $f\in C^3(U_{\e^{-\beta_{\ref{b4}}}})$, $g\in C^1(U_{\e^{-\beta_{\ref{b4}}}})$ 
 and $0<\e\leq \e_{\ref{e5}}$ satisfy
\[
-\Delta f+W'(f) = g
\]
on $U_{\e^{-\beta_{\ref{b4}}}}$ and
\[
\sup_{U_{\e^{-\beta_{\ref{b4}}}}} |f| \leq 1+\varepsilon^\eta, \ \ \ \sup_{U_{\e^{-\beta_{\ref{b4}}}}}
\left(
\frac{1}{2}|\nabla f|^2 -W(f)
\right)\leq c_{\ref{c11}}.
\]
Then
\begin{eqnarray}
\sup_{B_{\frac{\varepsilon^{-\beta_{\ref{b4}}}}{2}}} \left(
 \frac{1}{2}|\nabla f|^2 -W(f)
\right) \leq c_{\ref{c12}} (\varepsilon^{-\beta_{\ref{b4}}}\|g\|_{W^{1,n}(B_{\e^{-\beta_{\ref{b4}}}})}+\e^\eta).
\end{eqnarray}
\label{lemm9}
\end{lemm}
We remark that the assumptions on $W$ are essentially used for the proof of Lemma \ref{lemm9}.
\begin{proof}[Proof of Lemma \ref{lemm2}]
As in the proof of Lemma \ref{prereg}, 
define $\tilde{u}(x):=u(\e x)$, $\tilde{v}(x):=\e v(\e x)$,
 and subsequently drop $\tilde{\cdot}$ for simplicity. We have
\[
-\Delta u + W'(u)=  \nabla u \cdot v
\]
on $U_{\e^{-1}}$.
With respect to the new variables, we need to prove
\begin{eqnarray}
\sup_{U_{s\e^{-1}}} \left(
 \frac{1}{2}|\nabla u|^2 -W(u)
\right) \leq \varepsilon^{1-\beta_1}
\end{eqnarray}
for some $0<\beta_1<1$ for all sufficiently small $\varepsilon$.
Let $\phi_\lambda$ be the standard mollifier, namely, define
\begin{eqnarray}
\phi(x) :=
\begin{cases}
C \exp\left(\frac{1}{|x|^2-1} \right) & \mbox{ for }  |x|<1 \\
0 & \mbox{ for }  |x| \geq 1,
\end{cases}
\nonumber
\end{eqnarray}
where the constant $C>0$ is selected so that $\int_{\R^n} \phi=1$, and define $\phi_\lambda(x) :=\frac{1}{\lambda^n}\phi(\frac{x}{\lambda})$.
For $0<\beta_{\be\label{b5}}<1$ to be chosen depending only on $n$ and $p$ later, define for $x\in U_{\e^{-1}-1}$
\begin{eqnarray}
f(x):=(u*\phi_{\e^{\beta_{\ref{b5}}}})(x)=\int u(x-y)\phi_{\e^{\beta_{\ref{b5}}}}(y) \,dy.
\end{eqnarray}
By \eqref{u12a} and \eqref{hu12a}, we have
\begin{eqnarray}
\sup_{U_{\e^{-1}-1}}|f-u| \leq 2c_{\ref{c1}} \e^{\beta_{\ref{b5}}}, \label{65} \\
\sup_{U_{\e^{-1}-1}}|\nabla f -\nabla u|\leq 2c_{\ref{c1}} \varepsilon^{\beta_{\ref{b5}}(2-\frac{n}{p})}. \label{66}
\end{eqnarray}
We next define $g$ to be
\begin{eqnarray}
g:= (\nabla u \cdot v)*\phi_{\e^{\beta_{\ref{b5}}}}+(
W'(f)-W'(u)*\phi_{\e^{\beta_{\ref{b5}}}}), \label{67}
\end{eqnarray}
so that we have
\begin{eqnarray}
-\Delta f + W'(f)= g.
\end{eqnarray}
To use Lemma \ref{lemm9}, we next estimate the $W^{1,n}$ norm of $g$ on $U_{\e^{-\beta_{\ref{b4}}}}(x)$ with $x\in U_{s\e^{-1}}$,
where $0<\beta_4<1$ will be chosen depending only on $n$ and $p$. In the following, let 
us write $U_{\e^{-\beta_{\ref{b4}}}}(x)$ as $U_{\e^{-\beta_{\ref{b4}}}}$ and
$U_{\frac{\e^{-\beta_{\ref{b4}}}}{2}}(x)$ as $U_{\frac{\e^{-\beta_{\ref{b4}}}}{2}}$for simplicity.
For sufficiently small $\varepsilon$ depending on $s$ and $\beta_4$ so that $U_{2\varepsilon^{-\beta_4}}(x)
\subset U_{\varepsilon^{-1}-1}$ (and so that we may use \eqref{u12a}), the first term of (\ref{67})
 can be estimated as
\begin{eqnarray}
\| (\nabla u \cdot v)*\phi_{\e^{\beta_{\ref{b5}}}}\|_{W^{1,n}(U_{\e^{-\beta_4}})} \leq
c_{\cc\label{c13}}(1+\e^{-\beta_{\ref{b5}}})\|v\|_{L^n(U_{2\e^{-\beta_{\ref{b4}}}})} \label{69}
\end{eqnarray}
where $c_{\ref{c13}}$ depends only on $\phi$, $n$ and $c_{\ref{c1}}$. By \eqref{re-eq1}, we obtain
\begin{eqnarray}
\|v\|_{L^n(U_{2\e^{-\beta_{\ref{b4}}}})} &\leq& \|v\|_{L^{\frac{np}{n-p}}(U_{2\e^{-\beta_{\ref{b4}}}})}
\{\omega_n(2\e^{-\beta_{\ref{b4}}})^n \}^{\frac{2p-n}{np}} \nonumber \\
&\leq& \lambda_0 \e^{(2-\frac{n}{p})(1-\beta_{\ref{b4}}) } (2^n \omega_n)^{\frac{2p-n}{np}} .
\label{70}
\end{eqnarray}
Thus \eqref{69} and \eqref{70} show
\begin{eqnarray}
\| (\nabla u \cdot v)*\phi_{\e^{\beta_{\ref{b5}}}}\|_{W^{1,n}(U_{\e^{-\beta_{\ref{b4}}}})} 
\leq c_{\cc\label{c14}} \e^{(2-\frac{n}{p})(1-\beta_{\ref{b4}})-\beta_{\ref{b5}} } \label{71} 
\end{eqnarray}
where $c_{\ref{c14}}$ depends only on $\phi,n,p,\lambda_0$ and $c_{\ref{c1}}$. We next consider the
 second term of \eqref{67}. By \eqref{65}, \eqref{66} and
\[
W'(f)-W'(u)*\phi_{\e^{\beta_{\ref{b5}}}} = (W'(f)-W'(u))+(W'(u)-W'(u)*\phi_{\e^{\beta_{\ref{b5}}}}),
\]
we compute 
\begin{equation}
\sup|W'(f)-W'(u)| \leq \sup |W''|\,\sup|u-f|\,\leq c_{\ref{c15}}\e^{\beta_{\ref{b5}}}, 
\label{72}
\end{equation}
\begin{equation}
\begin{split}
\sup|\nabla(W'(f)-W'(u))| &\leq \sup|W''| \sup|\nabla f -\nabla u|  \\
&\,\,\,+ \sup|\nabla u| \sup |W'''| \sup|u-f|  \\
&\leq c_{\ref{c15}}\e^{\beta_{\ref{b5}}(2-\frac{n}{p})}, 
\end{split}
\end{equation}
\begin{equation}
 \sup|W'(u)-W'(u)*\phi_{\e^{\beta_{\ref{b5}}}} | \leq c_{\ref{c15}}\e^{\beta_{\ref{b5}}},
\end{equation}
\begin{equation}
\sup|\nabla (W'(u)-W'(u)*\phi_{\e^{\beta_{\ref{b5}}}}) | \leq c_{\ref{c15}}\e^{\beta_{\ref{b5}}(2-\frac{n}{p})}. \label{75}
\end{equation}
Here $c_{\cc\label{c15}}$ depends only on $\phi,n,\lambda_0,c_{\ref{c1}}$ and $W$.
Hence, (\ref{72})-(\ref{75}) show
\begin{eqnarray}
\|
W'(f)-W'(u)*\phi_{\e^{\beta_{\ref{b5}}}}\|_{W^{1,n}(U_{\e^{-\beta_{\ref{b4}}}})} \leq 4c_{\ref{c15}}
\e^{\beta_{\ref{b5}}(2-\frac{n}{p})-\beta_{\ref{b4}}}. \label{76}
\end{eqnarray}
By \eqref{67}, \eqref{71} and \eqref{76}, we have
\begin{eqnarray}
\|g\|_{W^{1,n}(U_{\e^{-\beta_{\ref{b4}}}})} \leq c_{\ref{c14}} \e^{(2-\frac{n}{p})(1-\beta_{\ref{b4}})-\beta_{\ref{b5}} } 
+4c_{\ref{c15}}
\e^{\beta_{\ref{b5}}(2-\frac{n}{p})-\beta_{\ref{b4}}}.
\end{eqnarray}
We use Lemma \ref{lemm9} to $f$ and $g$. Due to Lemma \ref{lemm8}, we have $\sup |f|\leq \sup |u|\leq 
1+\e^{\eta}$ on $U_{\e^{-\beta_{\ref{b4}}}}$ and we may choose smaller $\eta$ if necessary.
Because of \eqref{u12a}, we have $c_{\ref{c11}}\geq \sup_{U_{\e^{-\beta_{\ref{b4}}}}}(\frac{1}{2}|\nabla f|^2-W(f))$
for a constant depending only on $c_{\ref{c1}}$ and $W$ (here again we restrict $\varepsilon$ small so that
$U_{\varepsilon^{-\beta_4}}(x)\subset U_{\varepsilon^{-1}-1}\cap U_{\varepsilon^{-1}(1+s)/2}$). Then we have all the assumptions
for Lemma \ref{lemm9} and obtain
\begin{equation}
\begin{split}
\sup_{U_{\frac{\e^{-\beta_{\ref{b4}}}}{2}}}& \Big(
 \frac{1}{2}|\nabla f|^2 -W(f)
\Big) \\ & \leq c_{\ref{c12}}(\e^{(2-\frac{n}{p})(1-\beta_{\ref{b4}})-\beta_{\ref{b4}}-\beta_{\ref{b5}}}+\e^{\beta_{\ref{b5}}(2-\frac{n}{p})-2\beta_{\ref{b4}}}+\e^\eta).
\end{split}
\end{equation}
At this point, we fix sufficiently small $0<\beta_{\ref{b4}},\beta_{\ref{b5}}<1$
depending only on $n$ and $p$ such that
\[
(2-\frac{n}{p})(1-\beta_{\ref{b4}})-\beta_{\ref{b4}}-\beta_{\ref{b5}}>0, \ \ \ 
\beta_{\ref{b5}}(2-\frac{n}{p})-2\beta_{\ref{b4}}>0.
\]
This shows that we may choose a $0<\beta_{\ref{b1}}<1$ such that
\begin{eqnarray}
\sup_{U_{\frac{\e^{-\beta_{\ref{b4}}}}{2}}} \left(
 \frac{1}{2}|\nabla f|^2 -W(f)
\right) \leq \e^{1-\beta_{\ref{b1}}} \label{79}
\end{eqnarray}
for all sufficiently small $\e>0$.
We may take the center of $U_{\frac{\e^{-\beta_{\ref{b4}}}}{2}}(=U_{\frac{\e^{-\beta_{\ref{b4}}}}{2}}(x)$) to be any 
$x\in U_{s\e^{-1}}$ so that we have the estimate on $U_{s\e^{-1}}$. 
By \eqref{66}, \eqref{79} and
\[
\sup|W(f)-W(u)| \leq \sup |W'| \sup|u-f| \leq c_{\cc}\e^{\beta_{\ref{b5}}},
\]
we may also replace $f$ by $u$ in \eqref{79} by choosing a larger $0<\beta_{\ref{b1}}<1$ if
necessary. This proves the desired estimate. 
\end{proof}

\section{Rectifiability and integrality of the limit varifold}
In this section, we recover the index $i$ and assume that $\{u_i\}_{i=1}^{\infty}$,
$\{v_i\}_{i=1}^{\infty}$ and $\{\e_i\}_{i=1}^n$ satisfy \eqref{eq1}-\eqref{eq1d}. 
Define $\mu_i$ and $V_i$ as in \eqref{defmu} and \eqref{defvari}. By the 
standard weak compactness theorem of Radon measures, there exists a subsequence
(denoted by the same index) and a Radon measure $\mu$ and a varifold $V$ such that
\[
\mu_i\rightarrow \mu,\,\, \,V_i\rightarrow V. 
\]
\begin{lemm}
For $x \in {\rm spt}\,\mu$, there exists a subsequence $x_i \in U_1$ such that $u_i(x_i) \in [-\alpha,
\alpha]$ and $\lim_{i\to \infty}x_i =x$.  \label{lemm21}
\end{lemm}
\begin{proof}
We prove this by contradiction and assume that there exists some $r>0$ such that $|u_i|\geq \alpha$ on $U_r(x)$
 for all large $i$. Without loss of generality, assume $u_i \geq \alpha$ on $U_r(x)$. Then we repeat the same argument leading to (\ref{eq22a}) with $\phi$ there replaced by $C^1_c(U_r(x))$. The argument shows that $\lim_{i \rightarrow \infty} \int \frac{\e_i}{2}|\nabla u_i|^2 \phi^2=0$. Next, multiplying $u_i-1$ to the equation (\ref{eq1}) and using 
$W'(u_i)(u_i-1)\geq \frac{\kappa}{2}(u_i-1)^2$, we obtain 
\begin{equation}
\begin{split}
\int\frac{W}{\e_i}\phi^2 
&\leq c(W) \int\frac{(u_i-1)^2}{\e_i}\phi^2 
\leq \frac{2c(W)}{\kappa}\int\frac{W'(u_i)(u_i-1)}{\e_i}\phi^2 \\
&=\frac{2c(W)}{\kappa}\int (\e_i \Delta u_i (u_i -1)+\e_i (v_i \cdot \nabla u_i)(u_i-1))\phi^2. 
\end{split}
\label{2.2a}
\end{equation}
By integration by parts and (\ref{eq1d}), the right-hand side of (\ref{2.2a}) converges to $0$. This shows that $\mu(U_r(x))=0$ and contradicts $x\in{\rm spt}\, \mu$.
\end{proof}

\begin{theo} 
There exist constants $0<D_1 \leq D_2 < \infty$ which depend only on
 $c_0,\lambda_0, n,p,W,E_0$ and $s$ such that, for $x \in {\rm spt}\, \mu \cap U_{s}$
and $B_r(x)\subset U_{s}$, we have
\begin{equation}
D_1 r^{n-1} \leq \mu (B_r(x)) \leq D_2 r^{n-1}.
\end{equation}
\label{theo2.1}
\end{theo}
\begin{proof}
This follows immediately from Theorem \ref{theo5}, \ref{lemm6} and Lemma \ref{lemm21}.
\end{proof}
For the subsequent use, define
\[
\xi_i:=\frac{\e_i|\nabla u_i|^2}{2}-\frac{W(u_i)}{\e_i}.
\]
Once we have the monotonicity formula \eqref{eq41}, we may prove the following ``equi-partition of energy'' by
the same proof as in 
\cite[Proposition 4.3]{TY02}:
\begin{theo}
$\xi_i$, $\frac{\e_i}{2}|\nabla u_i|^2 - |\nabla w_i|$ and 
$\frac{W(u_i)}{\e_i}- |\nabla w_i|$ all converge to zero in $L^1_{loc}(U_1)$. \label{theo2.3}
\end{theo}
\begin{proof}[Proof of Theorem \ref{maintheorem}]
Recall that $\|V_i\|=\mu_i\res_{\{|\nabla u_i|\neq 0\}}$. 
Since $$\sigma \mu_i\res_{\{|\nabla u_i|=0\}}\leq |\xi_i|\,
d\mathcal L^n\to 0$$ 
by Theorem \ref{theo2.3},
$\mu_i\res_{\{|\nabla u_i|\neq 0\}}$ converges to $\mu$. We also know that
$\|V_i\|$ converges to $\|V\|$ by definition, thus we have
$\|V\|=\mu$. This proves (1). The claims (2) and (3) follows 
from Theorem \ref{theo5}, \ref{lemm6} and Lemma \ref{lemm21} (see also \cite[Proposition 4.2]{TY02}). 
 Next, by \eqref{defvari1}, \eqref{eq2} and Theorem \ref{theo2.3}, 
\begin{equation}
\begin{split}
\delta V_i(g)=&\int_{\{|\nabla u_i|\neq 0\}}{\rm div}\,g\,d\mu_i \\ &-
\frac{1}{\sigma}\int_{\{|\nabla u_i|\neq 0\}}\nabla g\cdot
\Big(\frac{\nabla u_i}{|\nabla u_i|}\otimes\frac{\nabla u_i}{|\nabla u_i|}\Big)(
\e_i|\nabla u_i|^2-\xi_i) \\
=&- \frac{1}{\sigma}\int \e_i(v_i\cdot\nabla u_i)(\nabla u_i\cdot g)\, +o(1)
\end{split} \label{2.5}
\end{equation} 
for $g \in C^1_c (U_1;\R^n)$, where $\lim_{i\to \infty}o(1)=0$. By Theorem \ref{theo5} and ${\rm spt}\,g \subset U_{s}$
for some $0<s<1$,
we have a uniform bound on 
$E(r,x)$ (corresponding to $u_i$)
for $B_r(x)\subset U_{s}$.
Hence, by Theorem \ref{MZ}, we have 
\begin{equation}
\begin{split}
\int &\e_i ((v_i-v) \cdot \nabla u_i)(\nabla u_i \cdot g)
\leq c\left( \int |v_i-v|^p |g|^p \e_i |\nabla u_i|^2 \right)^{\frac{1}{p}} \\
&\leq c\left( \int |\nabla(|v_i-v|^p |g|^p)| \right)^{\frac{1}{p}}  \\
&\leq c \left( \|\nabla v_i - \nabla v\|_{L^p} \|v_i-v\|_{L^p}^{p-1} +\|v_i-v\|_{L^p}^p  \right)^{\frac{1}{p}} 
\end{split}
\label{vapp2}
\end{equation}
where the integrations are over ${\rm spt}\,g$. The above converges to 0 since we may choose a further subsequence of $v_i$ which converges to $v$ strongly in $L^p_{loc}$. Thus in the right-hand side of \eqref{2.5}, we may 
replace $v_i$ by $v$. Let $\epsilon>0$ be arbitrary and let $\tilde v$ be a smooth
vector field such that $\|v-\tilde v\|_{W^{1,p}(U_{s})}<\epsilon$. 
By the varifold convergence $V_i \to V$, we have 
\begin{equation}
\begin{split}
\frac{1}{\sigma}\int \e_i (\tilde v \cdot \nabla u_i)(\nabla u_i \cdot g) &= \int S^{\perp}(\tilde v)\cdot g\, dV_i(x,S) +o(1) \\
&\rightarrow \int S^{\perp}(\tilde v)\cdot g\, dV(x,S).
\end{split}
\label{vapp}
\end{equation}
We may arbitrarily approximate the quantities in \eqref{vapp} by $v$ by the same argument 
in \eqref{vapp2}, 
hence by \eqref{2.5}-\eqref{vapp}, we obtain 
\begin{equation}
\delta V(g) = -\int S^{\perp}(v)\cdot g \, dV(x,S).
\label{delmean}
\end{equation}

Hence, $\|\delta V\|$ is a Radon measure on $U_1$. By (\ref{theo2.1}) and Allard's rectifiability theorem \cite[5.5.(1)]{WA72}, $V$
is rectifiable. Since $V$ is rectifiable, 
there exist an $\mathcal H^{n-1}$ measurable and countably $n-1$-rectifiable set $\Gamma$ such that 
\begin{equation}\label{delmeanad}
\int S^{\perp}(v)\cdot g\,dV(x,S)=-\int (T_x\,\Gamma)^{\perp}(v(x))\cdot g(x)\,d\|V\|(x).
\end{equation}
The set $\Gamma$ is the measure-theoretic support of $\|V\|$ and $T_x\,\Gamma$ is the 
approximate tangent space of $\Gamma$ which exists for $\mathcal H^{n-1}$ a.e.~on $\Gamma$. 
 Next from \eqref{delmean},  $\|\delta V\|$ is absolutely continuous with
 respect to $\|V\|$ and the generalized mean curvature $H_V$ exists. 
 By \eqref{delmean} and \eqref{delmeanad}, we have 
 $H_V(x)=(T_x\,\Gamma)^{\perp}(v(x))$ holds for $\|V\|$ a.e.~for $x$. 
 This proves (5), except that we do not yet take $\Gamma={\spt}\,\|V\|$. 
 The proof of (4) is the same as \cite{TY05} for the 
 following reason. We may
 set $f=\e \nabla u\cdot v$ in \cite{TY05} and we have
 $\|\e\nabla u\cdot v\|_{L^{\frac{np}{n-p}}(U_{s})}\leq c_{\ref{c1}}\lambda_0$
 due to Lemma \ref{prereg}. In the proof, as long as we have the monotonicity 
 formula \eqref{eq41} and the estimate Lemma \ref{lemm2}, all the argument
 goes through. The point is that we do not need to take a derivative of $f$
 for the proof of integrality and we only need the control of $L^{\frac{np}{n-p}}$ norm
 as well as the estimate \eqref{hu12a}. See the comment in the proof of \cite[Proposition 4.8]{TY02}
 where it is explained that \eqref{hu12a} is necessary. We should also point out that 
 the $L^n$ control of $\nabla f$ is not needed in the proof. 
 Finally, by arguing as in \eqref{vapp2} and the H\"{o}lder inequality, we have
\begin{equation*}
\int_{U_{s}} \phi^{\frac{p(n-1)}{n-p}} d\|V\| \leq c \|\nabla\phi\|_{L^p}
\|\phi\|_{L^{np/(n-p)}}^{n(p-1)/(n-p)}
\end{equation*}
for any function $\phi\in C^1_c(U_{s};\R^+)$ and we have the same 
inequality for $v\in W^{1,p}(U_1)$ by the density argument. Thus we have (6).
By the well-known property of varifold having the generalized mean curvature in $L^q$
with $q>n-1$ (see \cite[17.9(1)]{LS83}), ${\rm spt}\,\|V\|$ coincide with
the measure-theoretic support $\Gamma$, so that we have (5) with $\Gamma={\rm spt}\,\|V\|$.
This concludes the proof of Theorem \ref{maintheorem}.
\end{proof}
\section{Concluding remarks}
In \cite{LST1,KT16}, we studied the singular perturbation problem for
\begin{equation}
\partial_t u_{\e} +v_\e \cdot \nabla u_\e= \Delta u_\e - \frac{W'(u_\e)}{\e^2} \label{51}
\end{equation}
and proved that the time-parametrized family of limit varifolds satisfies 
the motion law of ``normal velocity $=$
mean curvavture vector + $v^{\perp}$'' in a weak formulation (see \cite{Ilmanen,TA} for
the case of 
$v_\e=0$). In these works, we assumed
that the prescribed initial data satisfies a boundedness of the 
upper density ratio. Part of
the difficulty was to show that the upper density ratio bound can be controlled locally in 
time and uniformly with respect to $\e$. For the equilibrium problem, it is certainly not natural to assume such an upper density ratio estimate. It is interesting to see if one can drop the upper
density ratio assumption for the initial data in the proof of \cite{LST1,KT16}. 

The vectorial prescribed mean curvature problem as in Theorem \ref{vmcp} 
seems, as far as we know, little studied
so far. Traditionally, the prescription is the scalar version, i.e. given a scalar function
(or constant) $f$, one looks for a hypersurface satisfying $H\cdot\nu=f$, where $\nu$ is the normal 
unit vector. The vectorial version is physically natural from the view point of force 
balance, in that the problem seeks the equality between the surface
tension force and an external force acting on the surface. It must be said that the
prescribed vector field in Theorem \ref{vmcp} is the gradient of a potential $\rho$, 
and not a general vector
field. This is rather restrictive for applications and it is interesting to know if there can 
be a remedy for generalizations. If there may not exist a variational framework such as
the min-max method to 
find solutions of \eqref{01}, it should be still useful to have this 
diffused interface approach since the functional is well-behaved functional-analytically.
As a further question, it is also interesting to investigate the asymptotic behavior of stable critical 
points of $F_{\e}$ in the proof of Theorem \ref{vmcp}, since we have a very
successful analogy in \cite{T05,TW12}. In this direction, we mention that a 
construction of prescribed scalar mean curvature
hypersurfaces along the lines suggested here has been announced recently \cite{Bel2} (see also \cite{Bel1}).

\end{document}